\documentclass[reqno,a4paper]{amsart}
\usepackage[active]{srcltx}
\usepackage{fancyhdr}
\usepackage[utf8]{inputenc}
\DeclareRobustCommand{\SkipTocEntry}[5]{} 
\usepackage{times}
\usepackage{latexsym,amsopn,amssymb,amsmath,mathabx}
\changenotsign
\usepackage{amsthm}
\usepackage{amsfonts,amsbsy,amscd,stmaryrd,bbold}
\usepackage{dsfont}
\usepackage{paralist}
\usepackage{hyperref}
\hypersetup{colorlinks=true,linkcolor=Emerald,urlcolor=blue} 

\usepackage[usenames,dvipsnames]{color}

\usepackage{mathrsfs} 

\usepackage{enumerate}

\newcommand{\C}{\mathbb{C}}
\newcommand{\N}{\mathbb{N}}
\newcommand{\R}{\mathbb{R}}

\newcommand{\cb}{\mathcal{B}}

\newcommand{\cd}{\mathcal{D}}
\newcommand{\ce}{\mathcal{E}}
\newcommand{\cf}{\mathcal{F}}
\newcommand{\cg}{\mathcal{G}}
\newcommand{\ch}{\mathcal{H}}

\newcommand{\ck}{\mathcal{K}}

\newcommand{\cs}{\mathcal{S}}

\newcommand{\cu}{\mathcal{U}}

\def\rond{\mathscr}

\newcommand{\rb}{\rond{B}}

\def\d{\mathrm{d}}

\def\rme{\mathrm{e}}

\def\i{\mathrm{i}}

\def\jap#1{\langle {#1} \rangle}

\DeclareMathOperator*{\w*lim}{w*-lim}

\def\qed{\hfill $\Box$\medskip}

\newtheorem{theorem}{Theorem}[section]
\newtheorem{lemma}[theorem]{Lemma}
\newtheorem{proposition}[theorem]{Proposition}
\newtheorem{corollary}[theorem]{Corollary}

\newtheorem{definition}[theorem]{Definition}

\numberwithin{equation}{section}

\begin{document}

\title[A new class of Schr\"odinger operators without positive eigenvalues]{A new class of Schr\"odinger operators without positive eigenvalues
  \\[2mm]
  {\tiny \today}
}

\author[A. Martin]{Alexandre Martin} \address{A. Martin,
  D\'epartement de Math\'ematiques, Universit\'e de Cergy-Pontoise,
  95000 Cergy-Pontoise, France}
\email{alexandre.martin@u-cergy.fr }

\begin{abstract}
Following the proof given by Froese and Herbst in \cite{FH} with another conjugate operator, we show for a class of real potential that possible eigenfunction of the Schr\"odinger operator has to decay sub-exponentially. We also show that, for a certain class of potential, this bound can not be satisfied which implies the absence of strictly positive eigenvalues for the Schr\"odinger operator.
\end{abstract}

\maketitle
\tableofcontents 

 \section{Introduction} \label{s:Intro}

In this article, we will study the Schr\"odinger operator $H=\Delta+V$ with a real potential, on $L^2(\R^\nu)$, where $\Delta$ is the non negative Laplacian operator. Here $V$ is a multiplication operator, i.e. $V$ can be the operator of multiplication by a real function or by a distribution of strictly positive order. When $V=0$, we know that $H=\Delta$ has a purely absolutly continuous spectrum on $[0,+\infty)$ with no embedded eigenvalues. We will try to see what happened if we add to $\Delta$ a "small" potential $V$, which is compact with respect to $\Delta$. In this case, $H$ is a compact perturbation of $\Delta$ and we already know that the essential spectrum of $H$ is $[0,+\infty)$.

An argument of quantum mechanics can make us believe that our Hamiltonians has no strictly positive eigenvalues, when $V$ is $\Delta$-compact or compact on $\ch^1$, the first order Sobolev space, to $\ch^{-1}$, its dual space.
This argument is reinforced by a result of S. Agmon \cite{Ag}, T. Kato \cite{Ka1}, R. Lavine \cite[Theorem XIII.29]{RS4}  and B. Simon \cite{Si1}. They proved the absence of positive eigenvalues for the operator $H=\Delta+V$ if the potential is a sum of a short range potential and a long range potential, i.e. $V$ can be written $V=V_1+V_2$ with 
\[
\begin{cases}
\lim\limits_{|x|\rightarrow +\infty} |x|V_1(x)=0 \\
\lim\limits_{|x|\rightarrow +\infty}V_2(x)=0\\
\lim\limits_{|x|\rightarrow+\infty} x\cdot\nabla V_2(x)\leq 0. 
\end{cases}\]
Similarly, L. H\"ormander \cite[Theorem 14.7.2]{Ho} proved that a possible eigenvector of $H$, associated to a positive eigenvalue, and its first order derivatives cannot have unlimited polynomial bounds if $|x|V$ is bounded. A.D. Ionescu and D. Jerison \cite{IJ} proved also this absence of positive eigenvalues for the 1-body Schr\"odinger operator, for a class of potentials with low regularity ($V\in L^{\nu/2}_{loc}$ if $\nu\geq3$, $V\in L^{r}_{loc}$, $r>1$ if $\nu=2$).

R. Froese, I. Herbst , M. Hoffman-Ostenhof and T. Hoffman-Ostenhof (\cite{FH} and \\
\cite{FHHH}) proved a similar result, concerning the N-body Schr\"odinger operator. We will explain below their result for the 1-body Schr\"odinger operator and we will generalize their proof to obtain larger conditions on the potential. More recently, using a similar proof than in \cite{FH}, two other results were proved. T. Jecko and A. Mbarek \cite{JMb} proved the absence of positive eigenvalues for $H=\Delta+V$  where $V$ is the sum of a short range potential, a long range potential and an oscillating potential which are not covered by the previous results. In the case of the discrete Schr\"odinger operator, M.A. Mandich \cite{Man} proved that under certain assumption on the potential, eigenfunctions decays sub-exponentially and that implies the absence of eigenvalues on a certain subset of the real axis.  This three proofs use the generator of dilations $A_D$, or the discrete generator of dilations in \cite{Man}, as conjugate operator. In our case, the continuous case, the generator of dilations has the following expression  
\[A_D=\frac{1}{2}(p\cdot q+q\cdot p),\]
where $q$ is the multiplication operator by $x$ and $p=-i\nabla$ is the derivative operator with $p^2=\Delta$. 

On the other hand, it is well known that we can construct a potential such that $H$ has positive eigenvalues.
For example, in one dimension, the Wigner-von Neuman potential $W(x)=w\sin(k|x|)/|x|$ with $k>0$ and $w\in \R^\varstar$ has a positive eigenvalue equal to $k^2/4$ (see \cite{NW}). Moreover, B. Simon proved in \cite{Si2} that for all sequence $(\ck_n)_{n=1\cdots +\infty}$ of distinct positive reals, we can construct a potential $V$ such that $(\ck_n^2)_{n=1\cdots +\infty}$ are eigenvalues of $H$. Moreover, B. Simon showed that if $\sum_{n=1}^\infty \ck_n<\infty$, then $|q|V$ is bounded, which implies that $V$ is $\Delta$-compact.

In their article \cite{FH}, R. Froese and I. Herbst proved the following
\begin{theorem}[\cite{FH}, Theorem 2.1]\label{th:FH}
Let $H=\Delta+V$ with $V$ a real-valued measurable function.
Suppose that
\begin{enumerate}
\item $V$ is $\Delta$-bounded with bound less than one,

\smallskip

\item $(\Delta+1)^{-1} q\cdot\nabla V(\Delta+1)^{-1}$ is bounded.
\end{enumerate}
Suppose that $H\psi=E\psi$. Then
 \begin{equation}
S_E=\sup\biggl\{\alpha^2+E\; ; \;\alpha>0, \exp(\alpha |x|)\psi\in L^2(\R^\nu)\biggr\}
\end{equation}
is $+\infty$ or the Mourre estimate is not valid at this energy with $A_D$ as conjugate operator.
\end{theorem}
From this result, they deduce the following
\begin{corollary}[\cite{FH}, Theorem 3.1]\label{c:FH}
Let $H=\Delta+V$ with $V$ a real-valued measurable function. Let $E>0$.
Suppose that
\begin{enumerate}
\item $V$ is $\Delta$-compact,

\smallskip

\item \label{itm: FH reg} $(\Delta+1)^{-1} q\cdot\nabla V(\Delta+1)^{-1}$ is compact,

\smallskip

\item for some $a<2$ and $b\in\R$, we have in the form sense
\begin{equation}\label{eq:FH}
q\cdot \nabla V\leq a\Delta+b.
\end{equation}
\end{enumerate}
Suppose that $H\psi=E\psi$.
Then $\psi=0$.
\end{corollary}

Following their proof, we will extend their result in two directions. First, we will see that for a larger class of $\Delta$-compact potential, we can prove that possible eigenvector of $H$ must satisfy some sub-exponential bounds in the $L^2$-norm. We will also show that this implies the absence of positive eigenvalue if $V$ satisfies a condition of type \eqref{eq:FH}.  Secondly, we will extend their results in the case where the potential is not $\Delta$-bounded but compact from $\ch^1$ to $\ch^{-1}$.
To prove these results we will use another conjugate operator of the form 
\[A_u=\frac{1}{2}(u(p)\cdot q+q\cdot u(p)),\]
where $u$ is a $C^\infty$ vector field with all derivatives bounded. Remark that this type of conjugate operator is essentially self-adjoint with the domain $\cd(A_u)\supset\cd(A_D)$ (see \cite[Proposition 4.2.3]{ABG}). This conjugate operator was also used in \cite{Ma1}. In this paper, it is proved that for a certain choice of $u$ ($u$ bounded), the commutator between $V$ and $A_u$ can avoid us to impose conditions on the derivatives of the potential, which can be useful when $V$ has high oscillations. Moreover, the commutator with the Laplacian, considered as a form with domain $\ch^1$, is quite explicit:
\[[\Delta,i A_u]=2p\cdot u(p)\]
which implies that the commutator is bounded from $\ch^1$ to $\ch^{-1}$. Since the unitary group generated by $A_u$ leaves invariant the domain and the form domain of the Laplacian (see \cite[Proposition 4.2.4]{ABG}), this proves that $\Delta$ is of class $C^1(A_u)$ and, similarly if if we add a potential $V$ which is $\Delta$-compact( respectively compact from $\ch^1$ to $\ch^{-1}$), with the regularity $C^1(A_u,\ch^2,\ch^{-2})$ (respectively $C^1(A_u,\ch^1,\ch^{-1})$), since the domain (respectively the form domain) is the same than the domain of the Laplacian, we deduce that $H=\Delta+V$ is of class $C^1(A_u)$.  If we take $u$ such that $x\cdot u(x)>0$ for all $x\not=0$, remark that the Mourre estimate is true with $A_u$ as conjugate operator on all compact subset of $(0,+\infty)$ for $\Delta$. For this reason, and to follow the proof of \cite{FH}, it will be convenient to choose $u$ of the form $x\lambda(x)$ with $\lambda:\R^\nu\rightarrow\R$ a positive function. All differences with \cite{FH} will be explain in Section \ref{s: exp bounds} and Section \ref{s:abs vp}.
 
 \section{Main results}
Now we will give our main results. Notice that we will recall the notion of regularity ($C^k$, $C^k_U$, $C^{1,1}$) with respect to an operator on Section \ref{ss:Regularity}.

To simplify notations, let $\cu$ be the set of vector fields $u$ with all derivatives bounded which can be writed $u(x)=x\lambda(x)$ with $\lambda$ a $C^\infty$ bounded positive function. In particular, $p\cdot\nabla\lambda(p)$ is bounded. We have the following:

\begin{theorem}\label{th:exp bound 1}
Let $H=\Delta +V$ on $L^2(\R^\nu)$, where $V$ is a symmetric potential such that $V$ is $\Delta$-bounded with bound less than one. Let $E\in\R$ and $\psi$ such that $H\psi=E\psi$. Suppose that there is $u\in \cu$ such that $(\Delta+1)^{-1}[V,iA_u](\Delta+1)^{-1}$ is bounded,
then, for all $0<\beta<1$,
\[S_E=\sup\biggl\{\alpha^2+E\; ; \;\alpha>0, \exp(\alpha\langle x\rangle^\beta)\psi \in L^2(\R^\nu)\biggr\}\]
is either $+\infty$ or in $\ce_u(H)$, the complement of the set of points for which the Mourre estimate (see Definition \ref{def: Mourre est}) is satisfied with respect to $A_u$.
\end{theorem}

We will give some comments about this Theorem:
\begin{enumerate}[\quad (a)]
\item Let $u\in\cu$. Since the unitary group generated by $A_u$ leaves invariant the Sobolev space $\ch^2$, $V\in C^1(A_u,\ch^2,\ch^{-2})$ if and only if $(\Delta+1)^{-1}[V,iA_u](\Delta+1)^{-1}$ is bounded. Thus, in this case, we can replace the assumption $(\Delta+1)^{-1}[V,iA_u](\Delta+1)^{-1}$ in Theorem \ref{th:exp bound 1} by an assumption of regularity.

\smallskip

\item Since we do not have an explicit expression for the commutator between an operator of multiplication and the conjugate operator $A_u$, in the proof of Theorem \ref{th:exp bound 1}, it is convenient to chose the function $F$, which appears in the proof, with a vanishing gradient at infinity. This is the case if $\beta<1$ but not if $\beta=1$. Remark that for certain type of potential, by using the interaction between the potential and $\Delta$, we can prove the exponential bounds or sub-exponential bounds ($\beta=1$), even if $V\notin C^1(A_u,\ch^2,\ch^{-2})$ (see \cite[Proposition 7.1]{JMb} and Proposition \ref{prop:oscillant}). 

\smallskip

\item Remark that if $V$ is $\Delta$-compact and $V\in C^1_\mathfrak{u}(A_u,\ch^2,\ch^{-2})$, for $u\in\cu$, $V$ satisfies assumptions of Theorem \ref{th:exp bound 1} and the Mourre estimate is true for all $\lambda\in(0,+\infty)$ (see \cite[Theorem 7.2.9]{ABG}). So, in this case, if $E>0$, then $\exp(\alpha\langle x\rangle^\beta)\psi\in L^2(\R^\nu)$ for all $\alpha>0$ and $\beta\in(0,1)$. Moreover, in this case, by the Virial Theorem, we can see 
 that the set of eigenvalues in $J=(0,+\infty)$ has no accumulation point inside $J$ and are of finite multiplicity.

\smallskip

\item If $V$ vanishes at infinity and can be seen as the Laplacian of a short range potential (i.e. $V=\Delta W$ with $\lim\limits_{|x|\rightarrow +\infty}\jap{x}W=0$), then $V$ is $\Delta$-compact and $\jap{q}V:\ch^2\rightarrow\ch^{-2}$ is compact. In this case, we can apply Theorem \ref{th:exp bound 1} to $H=\Delta+V$.

\smallskip

\item For $\zeta\in\R$, $\theta>0$, $k\in\R^*$ and $w\in\R$, let
\[V(x)=w(1-\kappa(|x|))\frac{\sin(|x|^\zeta)}{|x|^\theta},\]
 with $\kappa\in C^\infty_c(\R,\R)$ with $\kappa(|x|)=1$ if $|x|<1$, $0\leq\kappa\leq1$. Note that this type of potential was already studied in \cite{BaD, DMR, DR1, DR2, JMb, ReT1, ReT2}. If $\zeta<\theta$ or if $\theta>1$, we can see that $V$ is a long range  or a short range potential. Moreover, in \cite{JMb}, it is proved that if $\zeta+\theta>2$, then $V$ has a good regularity with respect to $A_D$.  So we can apply Theorem \ref{th:FH} in these two areas. In \cite{JMb}, they also showed that if $\zeta>1$ and $\theta>1/2$, then a possible eigenvector associated with positive energy has unlimited exponential bounds. But, if $|\zeta-1|+\theta<1$, they proved that $H\notin C^1(A_D)$ and so we cannot apply Theorem \ref{th:FH} with this potential.
 If $2\zeta+\theta>3$, $\zeta>1$ and $0<\theta\leq1/2$, then $V$ is of class $C^{1,1}(A_u,\ch^2,\ch^{-2})\subset C^1_\mathfrak{u}(A_u,\ch^2,\ch^{-2})$ for all $u$ bounded (see \cite[Lemma 5.4]{Ma1}). So, Theorem \ref{th:exp bound 1} applies if $2\zeta+\theta>3$ with $\zeta>1$ and $0<\theta\leq1/2$.
 
\end{enumerate}





Since the Laplacian operator $\Delta$ can be seen as a form on $\ch^1$, the first order Sobolev space, to $\ch^{-1}$, the dual space of $\ch^1$, we can also study the case where $V:\ch^1\rightarrow\ch^{-1}$ is compact. In this case, the difference between the resolvent of $H$ and the resolvent of $\Delta$ is compact and the essential spectrum of $H$ is still $[0,+\infty)$. We have the following

\begin{theorem}\label{th:exp bound 2}
Let $H=\Delta +V$ on $L^2(\R^\nu)$, where $V$ is a real-valued function such that $V:\ch^1\rightarrow\ch^{-1}$ is bounded with relative bound less than one. Let $E\in\R$ and $\psi$ such that $H\psi=E\psi$. If there is $u\in \cu$ such that $\jap{p}^{-1}[V,iA_u]\jap{p}^{-1}$ is bounded,
then, for all $0<\beta<1$,
\[S_E=\sup\biggl\{\alpha^2+E\; ; \;\alpha>0, \exp(\alpha\langle x\rangle^\beta)\psi \in L^2(\R^\nu)\biggr\}\]
is either $+\infty$ or in $\ce_u(H)$.
\end{theorem}
We make some comments about this theorem:
\begin{enumerate}[\quad (a)]
\item Since $V:\ch^1\rightarrow\ch^{-1}$ is bounded with relative bound less than one, by the KLMN Theorem, $H$ can be considered as a form with form domain $\ch^1$ and is associated to a self-adjoint operator.

\smallskip

\item Let $u\in\cu$. Since the unitary group generated by $A_u$ leaves invariant the Sobolev space $\ch^1$, $V\in C^1(A_u,\ch^1,\ch^{-1})$ if and only if 
\[(\Delta+1)^{-1/2} [V,iA_u](\Delta+1)^{-1/2}\]
 is bounded. Thus, in this case, we can replace the assumption on \[(\Delta+1)^{-1/2} [V,iA_u](\Delta+1)^{-1/2}\] in Theorem \ref{th:exp bound 2} by an assumption of regularity.

\smallskip

\item If $V:\ch^1\rightarrow\ch^{-1}$ is compact and if $V\in C^1_\mathfrak{u}(A_u,\ch^1,\ch^{-1})$, then 
\[(\Delta+1)^{-1/2} [V,iA_u](\Delta+1)^{-1/2}\]
 is compact. Thus the Mourre estimate is true on all compact subset of $(0,+\infty)$. So, if $E>0$, in this case, the sub-exponential bounds are true for all $\alpha>0$.

\smallskip

\item For $\zeta\in\R$, $\theta\in\R$, $k\in\R^*$ and $w\in\R$, let
\[V(x)=w(1-\kappa(|x|))\frac{\sin(|x|^\zeta)}{|x|^\theta},\]
 with $\kappa\in C^\infty_c(\R,\R)$ with $\kappa(|x|)=1$ if $|x|<1$, $0\leq\kappa\leq1$. If $\zeta+\theta>2$, then $V:\ch^1\rightarrow \ch^{-1}$ is compact and $V$ is of class $C^{1,1}(A_u,\ch^1,\ch^{-1})\subset C^1_\mathfrak{u}(A_u,\ch^1,\ch^{-1})$ for all $u$ bounded (see \cite[Lemma 5.4]{Ma1}). So, Theorem \ref{th:exp bound 2} applies if $\zeta+\theta>2$, even if $\theta\leq0$.

 \smallskip
 
 \item Let \[V(x)=w(1-\kappa(|x|)) e^{3|x|/4}\sin(e^{|x|})\]
with $w\in\R$, $\kappa\in C^\infty_c(\R,\R)$, $0\leq\kappa\leq1$ and $\kappa(|x|)=1$ if $|x|<1$. Note that this type of potential was already studied in \cite{Co,CoG}. We can show that $V:\ch^1\rightarrow \ch^{-1}$ is compact and $V$ is of class $C^{1,1}(A_u,\ch^1,\ch^{-1})\subset C^1_\mathfrak{u}(A_u,\ch^1,\ch^{-1})$ for all $u$ bounded (see\cite[Lemma 5.6]{Ma1}). So, for all $w\in\R$, Theorem \ref{th:exp bound 2} applies. Moreover, since $V$ is not $\Delta$-bounded, we cannot apply Theorem \ref{th:FH}.

\smallskip

\item Assume that $V:\ch^1\rightarrow\ch^{-1}$ is symmetric, bounded with bound less than one and that there is $\mu>0$ such that $\jap{x}^{1+\mu}V(x)\in\ch^{-1}$. Then there is $u\in\cu$ such that $V\in C^{1,1}(A_u,\ch^1,\ch^{-1})$ (see \cite[Lemma 5.8]{Ma1}). In particular, for this type of potential, Theorem \ref{th:exp bound 2} applies. For example, in dimension $\nu\geq3$, if we take $\chi:\R\rightarrow\R$ such that $\chi\in C^3$, $\chi(|x|)=0$ if $|x|>1$ and $\chi'(0)=\chi''(0)=1$, the potential defined by 
\[V(x)=\sum\limits_{n=2}^{+\infty}n^{(3\nu-1)/2}\chi'(n^{3\nu/2}(|x|-n)),\] 
is compact on $\ch^1$ to $\ch^{-1}$ and of class $C^{1,1}(A_u,\ch^1,\ch^{-1})$ for an appropriate $u$. Moreover, we can show that this potential is neither $\Delta$-bounded, neither of class $C^1(A_D,\ch^1,\ch^{-1})$ (see \cite[Lemma 5.10]{Ma1}). In particular, Theorems \ref{th:FH} and \ref{th:exp bound 1} do not apply with this potential.

\smallskip

\item Remark that all examples we gave are central potentials. But it is not necessary to have this property and we gave only examples which are central because it is easier. In particular, if $W$ satisfies $\jap{q}^{1+\epsilon}W$ is bounded for one choice of $\epsilon>0$, then Theorem \ref{th:exp bound 2} applies for $V=\text{div}(W)$.
\end{enumerate}

Since in the proof of Corollary \ref{c:FH}, one use only assumption \eqref{eq:FH} by applying it on certain vectors that are constructed with a possible eigenvector of $H$, we can weaken the conditions on the potential.
 For $0<\beta<1$ and $\alpha>0$, let $F_\beta(x)=\alpha\jap{x}^\beta$. We have the following

\begin{theorem}\label{th:abs 1}
Suppose that $V$ is $\Delta$-compact.\\
Let $\psi$ such that $H\psi=E\psi$ with $E>0$ and such that $\psi_F=\exp(F_\beta(q))\psi\in L^2(\R^\nu)$ for all $\alpha>0$, $0<\beta<1$. \\
Suppose that there is $\delta>-2$, $\delta',\sigma,\sigma'\in\R$ such that $\delta+\delta'>-2$ and, for all $\alpha>0$, $0<\beta<1$,
\begin{equation}\label{eq:V,iA petit 1}
(\psi_F,[V,iA_D]\psi_F)\geq\delta(\psi_F,\Delta\psi_F)+\delta' (\psi_F,(\nabla F_\beta)^2\psi_F)+(\sigma\alpha+\sigma')\|\psi_F\|^2.
\end{equation}
Then $\psi=0$.
\end{theorem}

If we only suppose that $V:\ch^1\rightarrow\ch^{-1}$ is bounded (but not necessarily $\Delta$-bounded), we have the following:
\begin{theorem}\label{th:abs 2}
Suppose that $V:\ch^1\rightarrow\ch^{-1}$ is bounded.\\
Let $\psi$ such that $H\psi=E\psi$ with $E>0$. For $0<\beta<1$ and $\alpha>0$, let $F_\beta(x)=\alpha\jap{x}^\beta$. Denote $\psi_F=\exp(F_\beta(q))\psi$. \\
Suppose that $\psi_F\in L^2(\R^\nu)$ for all $\alpha>0$, $0<\beta<1$, and that there is $\delta>-2$, $\delta',\sigma,\sigma'\in\R$ such that $\delta+(1+\|\jap{p}^{-1}V\jap{p}^{-1}\|)\delta'>-2$ and, for all $\alpha>0$, $0<\beta<1$,
\begin{equation}\label{eq:V,iA petit 2}
(\psi_F,[V,iA_D]\psi_F)\geq\delta(\psi_F,\Delta\psi_F)+\delta' (\psi_F,(\nabla F_\beta)^2\psi_F)+(\sigma\alpha+\sigma')\|\psi_F\|^2.
\end{equation}
Then $\psi=0$.
\end{theorem}




We make some comments on the two previous theorems:

\begin{enumerate}[\quad (a)]
\item Since we suppose that $\psi$ has sub-exponential bounds, for $\alpha,\beta$ fixed, $\psi_F$ has sub-exponential bounds too. Moreover, we can remark that $\psi_F$ is an eigenvector for 
\[H(F):=e^F H e^{-F}= H-(\nabla F)^2+(ip\nabla F+i\nabla F p).\]
This makes easier to prove \eqref{eq:V,iA petit 1} and \eqref{eq:V,iA petit 2}.

\smallskip

\item Remark that in \eqref{eq:FH}, the inequality is required to be true in the sense of the form. In \eqref{eq:V,iA petit 1} and \eqref{eq:V,iA petit 2}, we do not ask to have this inequalities for all $\phi\in\cd(H)\cap\cd(A_u)$, but only for a type of vector with high decrease at infinity.
 
\smallskip

\item Assumption \eqref{eq:FH} corresponds to the case where $\delta'=\sigma=0$ and $\delta>-2$ in \eqref{eq:V,iA petit 1}. In particular, if $V$ satisfies \eqref{eq:FH}, it satisfies \eqref{eq:V,iA petit 1} too.

\smallskip

\item Remark that if $\delta'\geq0$, conditions $\delta+\delta'>-2$ and $\delta+(1+\|\jap{p}^{-1}V\jap{p}^{-1}\|)\delta'>-2$ are always satisfied.

\smallskip

\item Actually, one only need to require \eqref{eq:V,iA petit 1} and/or \eqref{eq:V,iA petit 2} for $\beta$ near $1$ and $\alpha$ large enough.

\smallskip

\item We can replace \eqref{eq:V,iA petit 1} by the similar inequality
\begin{multline}\label{eq:V,iA petit 1'}
(\psi_F,[V,iA_D]\psi_F)\geq\delta(\psi_F,\Delta\psi_F)+\delta' (\psi_F,(\nabla F_\beta)^2\psi_F)+\delta''\|g^{1/2}A_D\psi_F\|^2\\
+(\sigma\alpha+\sigma')\|\psi_F\|^2\tag{\ref{eq:V,iA petit 1}'}
\end{multline}
with $\delta>-2$, $\delta+\delta'>-2$ and $\delta''>-4$. \eqref{eq:V,iA petit 2} may be replaced by
\begin{multline}\label{eq:V,iA petit 2'}
(\psi_F,[V,iA_D]\psi_F)\geq\delta(\psi_F,\Delta\psi_F)+\delta' (\psi_F,(\nabla F_\beta)^2\psi_F)+\delta''\|g^{1/2}A_D\psi_F\|^2\\
+(\sigma\alpha+\sigma')\|\psi_F\|^2\tag{\ref{eq:V,iA petit 2}'}
\end{multline} 
with $\delta>-2$, $\delta+(1+\|\jap{p}^{-1}V\jap{p}^{-1}\|)\|\delta'>-2$ and $\delta''>-4$ and the both Theorems remain true. This enlarges the class of admissible potentials (see Section \ref{s:examples}).

\smallskip

\item Let $V_{sr}$ and $V_{lr}$ be two functions such that there is $\rho_{sr},\rho_{lr},\rho'_{lr}>0$ and\\
 $|x|^{1+\rho_{sr}}V_{sr}(x)$, $|x|^{\rho_{lr}}V_{lr}(x)$ and $|x|^{\rho'_{lr}}x\nabla V_{lr}(x)$ are bounded. Suppose that $V$ satisfies assumptions of Theorem \ref{th:abs 1} (respectively Theorem \ref{th:abs 2}). \\
 Then $\tilde{V}=V+V_{sr}+V_{lr}$ satisfies assumptions of Theorem \ref{th:abs 1} (respectively Theorem \ref{th:abs 2}) too. To see that, notice that $V_{lr}$ and $V_{sr}$ are compact on $\ch^1$ and are of class $C^1(A_D,\ch^1,\ch^{-1})\cap C^1_\mathfrak{u}(A_u,\ch^1,\ch^{-1})$ for all $u\in\cu$ and that there is $\sigma_1,\sigma_2\in\R$ such that 
\begin{eqnarray*}
(\psi_F,[V_{lr},iA_D]\psi_F)&\geq& \sigma_1\|\psi_F\|^2 \\
 (\psi_F,[V_{sr},iA_D]\psi_F)&\geq&  -\epsilon (\psi_F,\Delta\psi_F)+\frac{\sigma_2}{\epsilon}\|\phi\|^2
 \end{eqnarray*}
 for all $\epsilon>0$. In particular, we can choose $\epsilon>0$ small enough such that, if $V$ satisfies \eqref{eq:V,iA petit 1} (respectively \eqref{eq:V,iA petit 2}), $\tilde{V}$ satisfies \eqref{eq:V,iA petit 1} (respectively \eqref{eq:V,iA petit 2}).
 
 \smallskip
 
 \item If $V$ can be seen as the derivative of a bounded function (the derivative of a short range potential for example), the conclusion of Theorem \ref{th:abs 2} is still true if one assume \eqref{eq:V,iA petit 2} and if one replaces the condition $\delta+(1+\|\jap{p}^{-1}V\jap{p}^{-1}\|)\delta'>-2$ by the weaker condition $\delta+\delta'>-2$.

\smallskip

\item For $\zeta,\theta\in\R$, $k\in\R^*$ and $w\in\R$, let
\[V(x)=w(1-\kappa(|x|))\frac{\sin(|x|^\zeta)}{|x|^\theta},\]
 with $\kappa\in C^\infty_c(\R,\R)$ with $\kappa(|x|)=1$ if $|x|<1$, $0\leq\kappa\leq1$. As for the sub-exponential bounds, we can see that if $\theta>0$ and $\zeta<\theta$ or $\theta>1$, then Corollary \ref{c:FH} applies. In \cite{JMb}, they showed that if $\zeta>1$ and $\theta>1/2$, $H=\Delta+V$ has no positive eigenvalues. Moreover, they claimed that if $\theta>0$, $\zeta+\theta>2$ and $|w|$ is small enough then $V$ satisfies \eqref{eq:FH} and so Corollary \ref{c:FH} applies. But their proof is not sufficient if $\theta\leq 1/2$ because we need to have the commutator bounded from $\ch^1$ to $\ch^{-1}$ and in this case, it is only bounded from $\ch^2$ to $\ch^{-2}$. Here, we can show a better result: if $\zeta+\theta>2$, $V:\ch^1\rightarrow\ch^{-1}$ is compact, of class $C^{1,1}(A_u,\ch^1,\ch^{-1})$ for all u bounded and satisfies \eqref{eq:V,iA petit 2} for all $w$. Therefore $V$ satisfies assumptions of Theorem \ref{th:abs 2} for all $w\in\R$. In particular, if $\theta<0$, $V$ is not bounded. Moreover, if $\zeta+\theta=2$ and $1/2\geq\theta$, then $V$ satisfies assumptions of Theorem \ref{th:abs 1} for $|w|$ sufficiently small. All this results are collected in Proposition \ref{prop:oscillant}.
 
 \smallskip
 
 \item Let \[V(x)=w(1-\kappa(|x|)) e^{3|x|/4}\sin(e^{|x|})\]
with $w\in\R$, $\kappa\in C^\infty_c(\R,\R)$, $0\leq\kappa\leq1$ and $\kappa(|x|)=1$ if $|x|<1$. For all $w\in\R$,  we can apply Theorem \ref{th:abs 2} (see Lemma \ref{l:ex expo}). Moreover, since $V$ is not $\Delta$-bounded, we cannot apply Corollary \ref{c:FH}.
\end{enumerate}

Now, we assume that $V$ has more regularity with respect to $A_u$. In this case, we can prove a limiting absorption principle and we can show that the boundary values of the resolvent will be a smooth function outside the eigenvalues. To this end, we need to use the
\emph{H\"older-Zygmund continuity classes} denoted $\Lambda^\sigma$. The definition of this particular classes of regularity is recalled on Section \ref{ss:Regularity}. We also need some weighted Sobolev space, denoted $\ch^t_s$ which are defined on Section \ref{ss:Notation}
\begin{theorem}[\cite{Ma1}, Theorem (3.3)]\label{th: Theorem 3.3}
Let $R(z)=(H-z)^{-1}$ be the resolvent operator associate to $H$.
 Let $V:\ch^1\to\ch^{-1}$ be a compact
  symmetric operator.  Suppose that there is $u\in\cu$ and $s>1/2$
  such that $V$ is of class
  $\Lambda^{s+1/2}(A_u,\ch^1,\ch^{-1})$. Then the limits
\begin{equation}\label{eq:lap}
R(\lambda\pm\i0):=\w*lim_{\mu\downarrow0}R(\lambda\pm i \mu)
\end{equation}
exist, locally uniformly in
$\lambda\in (0,+\infty)$ outside the eigenvalues of $H$. Moreover, the functions
\begin{equation}\label{eq:reg}
\lambda\mapsto R(\lambda\pm\i0)\in B(\ch^{-1}_s,\ch^1_{-s})
\end{equation}
are locally of class $\Lambda^{s-1/2}$ on $(0,+\infty)$ outside the eigenvalues
of $H$. 
\end{theorem}
 Since $\Lambda^{s+1/2}(A_u)\subset C^1_\mathfrak{u}(A_u)$ for all $s>1/2$, by combining Theorems \ref{th:exp bound 2}, \ref{th:abs 2} and \ref{th: Theorem 3.3}, we have the following
\begin{corollary}\label{c:resolvent}
Let $V:\ch^1\rightarrow\ch^{-1}$ be a compact symmetric potential and $s>1/2$. If there is $u\in \cu$ such that $V$ is of class $\Lambda^{s+1/2}(A_u,\ch^1,\ch^{-1})$, and if \eqref{eq:V,iA petit 2} is satisfied, then the limits
\begin{equation}
R(\lambda\pm\i0):=\w*lim_{\mu\downarrow0}R(\lambda\pm i \mu)
\end{equation}
exist locally uniformly in
$\lambda\in (0,+\infty)$ and
\begin{equation}
\lambda\mapsto R(\lambda\pm i0)\in B(\ch^{-1}_s,\ch^1_{-s})
\end{equation}
are of class $\Lambda^{s-1/2}$ on $(0,+\infty)$.
\end{corollary}

 
 The paper is organized as follows. In Section \ref{s:notation}, we will give some notations and we recall some basic fact about regularity. In Section \ref{s: exp bounds}, we will prove Theorem \ref{th:exp bound 1} and  Theorem \ref{th:exp bound 2}. In Section \ref{s:abs vp}, we will prove Theorem \ref{th:abs 1} and Theorem \ref{th:abs 2}. In Section \ref{s:examples}, we will give some explicit classes of potential for which we can apply our main results. Finally in Appendix \ref{s: H-S formula}, we will recall the Helffer-Sj\"ostrand formula and some properties of this formula that we will use in the proof of our main Theorems.
   
\section{Notations and basic notions}\label{s:notation}

\subsection{Notation}\label{ss:Notation} Let $X=\R^\nu$ and for $s\in\R$ let $\ch^s$ be the usual Sobolev space on $X$ with $\ch^0=\ch=L^2(X)$ whose norm is denoted $\|\cdot\|$. We are mainly interested in the space $\ch^{1}$ defined
by the norm $\|f\|_1^2=\int\left(|f(x)|^2+|\nabla f(x)|^2\right)\d x$ and
its dual space $\ch^{-1}$.

We denote $q_j$ the operator of multiplication by the coordinate $x_j$
and $p_j=-\i\partial_j$ considered as operators in $\ch$. For $k\in X$ we denote
$k\cdot q=k_1q_1+\dots+k_\nu q_\nu$. If $u$ is a
measurable function on $X$ let $u(q)$ be the operator of
multiplication by $u$ in $\ch$ and $u(p)=\cf^{-1}u(q)\cf$, where $\cf$
is the Fourier transformation:
\[
(\cf f)(\xi)=(2\pi)^{-\frac{\nu}{2}} \int \rme^{-i  x\cdot\xi} u(x) \d x .
\]
If there is no ambiguity we keep the same notation for these operators
when considered as acting in other spaces.  
If $u$ is a $C^\infty$ vector fields with all the derivates bounded, we denote by $A_u$ the symmetric operator:

\begin{equation}
A_u=\frac{1}{2}(q\cdot u(p)+u(p)\cdot q)=u(p)\cdot q+\frac{i }{2}(\mathrm{div} u)(p).
\end{equation}
Notice that $A_u$ is essentially self-adjoint (see \cite[Proposition 4.2.3]{ABG}). Since we will use vector fields $u$ which have a particular form, we use the space $\cu$ define by
\begin{definition}\label{def: classe U}
We define $\cu$ the space of $C^\infty$ vector fields $u$ with all derivates bounded such that there is a  strictly positive bounded function $\lambda:X\rightarrow\R$ of class $C^\infty$  with $u(x)=x\lambda(x)$ for all $x\in X$.
\end{definition}

Let $
A_D=\frac{1}{2}(p\cdot q+q\cdot p)
$ be the generator of dilations.

As usual, we denote  $\langle x\rangle=(1+|x|^2)^{1/2}$. Then $\jap{q}$ is the operator of
multiplication by the function $x\mapsto\jap{x}$ and
$\jap{p}=\cf^{-1}\jap{q}\cf$.  For real $s,t$ we denote $\ch^t_s$ the
space defined by the norm
\begin{equation}\label{eq:K}
\|f\|_{\ch^t_s}= \|\jap{q}^s f\|_{\ch^t}= \|\jap{p}^t \jap{q}^s f\|.
\end{equation}
Note that the norm $\|f\|_{\ch^t_s}$ is equivalent to the norm $\| \jap{q}^s\jap{p}^t f\|$ and that the adjoint space of $\ch^t_s$ may be identified with
$\ch^{-t}_{-s}$.

\smallskip

We denote $\Delta=p^2$ the non negative Laplacian operator, i.e. for all $\phi\in \ch^2$, we have 
\[\Delta \phi=-\sum_{i=1}^n \frac{\partial^2 \phi}{\partial x_i^2}.\]

For $I$ a Borel subset of $\R$, we denote $E(I)$ the spectral mesure of $H$ on $I$.

\begin{definition}\label{def: Mourre est}
Let $A$ be a self adjoint operator on $L^2(\R^\nu)$. Assume that $H$ is of class $C^1(A)$.
We say that $H$ satisfies the Mourre estimate at $\lambda_0$ with respect to the conjugate operator $A$ if there exists a non-empty open set $I$ containing $\lambda_0$, a real $c_0>0$ and a compact operator $K_0$ such that 

\begin{equation}\label{eq:Mourre estimate}
E(I)[H,iA]E(I)\geq c_0 E(I)+K_0
\end{equation}
\end{definition}

We denote $\ce_u(H)$ the complement of the set of $\lambda_0$ for which the Mourre estimate is satisfied with respect to $A_u$.

In the Helffer-Sj\"ostrand formula (Appendix \ref{s: H-S formula}), there is a term of rest which appears. To control it we define the following space of application:
\begin{definition}\label{def:S rho}
For $\rho\in\R$, let $\cs^\rho$ be the class of the function $\varphi\in C^\infty(\R^\nu,\C)$ such that

\begin{equation}
\forall k\in\N, \qquad C_k(\varphi)\coloneq \sup_{ \substack{t\in\R^\nu  \\ |\alpha|=k}} \langle t\rangle^{-\rho+k}|\partial^\alpha_t \varphi(t)|<\infty.
\end{equation}
\end{definition}

Note that $C_k$ define a semi-norm for all $k$.

\subsection{Regularity}\label{ss:Regularity}

Let $F', F''$ be to Banach space and $T:F'\rightarrow F''$ a bounded operator.

Let $A$ a self-adjoint operator such that the unitary group generated by $A$ leaves $F'$ and $F''$ invariants.

Let $k\in\N$. We said that $T\in C^k(A,F',F'')$ if, for all $f\in F'$, the map \\
$\R\ni t\mapsto e^{itA}Te^{-itA}f$ has the usual $C^k$ regularity. 

We said that $T\in C^k_\mathfrak{u}(A,F',F'')$ if $T\in C^k(A,F',F'')$ and all the derivatives of the map $\R\ni t\mapsto e^{itA}Te^{-itA}f$ are norm-continuous function. The following characterisation is available:

\begin{proposition}[Proposition 5.1.2, \cite{ABG}]
 $T\in C^1(A,F',F'')$ if and only if $[T,A]=TA-AT$ has an extension in $\cb(F',F'')$.
\end{proposition}

For $k>1$, $T\in C^k(A,F',F'')$ if and only if $T\in C^1(A,F',F'')$ and $[T,A]\in C^{k-1}(A,F',F'')$.


We can defined another class of regularity called the $C^{1,1}$ regularity:

\begin{proposition}
We said that $T\in C^{1,1}(A,F',F'')$ if and only if 
\[
\int_0^1\|T_\tau+T_{-\tau}-2T\|_{\cb(F',F'')} \,\frac{\d\tau}{\tau^2}<\infty, 
\]
 where $T_\tau=\rme^{\i\tau A_u}T\rme^{-\i\tau A_u}$.
\end{proposition} 

An easier result can be used:
\begin{proposition}[Proposition 7.5.7 from \cite{ABG}]
Let $A$ be a self-adjoint operator. Let $\cg$ be a Banach space and let $\Lambda$ be a closed densely defined operator in $\cg^*$ with domain included in $\cd(A,\cg^*)$ and such that $-i r$ belongs to the resolvent set of $\Lambda$ and $r\|(\Lambda+i r)^{-1}\|_{\cb(\cg^*)}\leq C\in\R$ for all $r>0$. Let $\xi\in C^\infty(X)$ such that \\$\xi(x)=0$ if
$|x|<1$ and $\xi(x)=1$ if $|x|>2$.
If $T:\cg\rightarrow\cg^*$ is symmetric, of class $C^1(A,\cg,\cg^*)$ and satisfies
\[\int_1^\infty \|\xi(\Lambda/r)[T,iA]\|_{\cb(\cg,\cg^* )}
  \frac{\d r}{r} <\infty\]
  then $T$ is of class $C^{1,1}(A,\cg,\cg^*)$.
\end{proposition} 

If $T$ is not bounded, we said that $T\in C^k(A,F',F'')$ if for $z\notin\sigma(T)$, \\$(T-z)^{-1}\in C^k(A,F'',F')$.

\begin{proposition}
For all $k>1$, we have
\[ C^{k}(A,F',F'')\subset C^{1,1}(A,F',F'')\subset C^1_\mathfrak{u}(A,F',F'')\subset C^{1}(A,F',F'').\]
\end{proposition}
If $F'=F''=\ch$ is an Hilbert space, we note $C^1(A)=C^1(A,\ch,\ch^*)$.
If $T$ is self-adjoint, we have the following:

\begin{theorem}[Theorem 6.3.4 from \cite{ABG}]\label{th:634}
Let $A$ and $T$ be two self-adjoint operators in a Hilbert space $\ch$. Assume that the unitary group $\{\exp(iA\tau)\}_{\tau\in\R}$ leaves the domain $D(T)$ of $T$ invariant. Set $\cg=D(T)$. Then 
\begin{enumerate}
\item $T$ is of class $C^1(A)$ if and only if $T\in C^1(A,\cg,\cg^*)$.

\item $T$ is of class $C^{1,1}(A)$ if and only if $T\in C^{1,1}(A,\cg,\cg^*)$.
\end{enumerate}
\end{theorem}
Remark that, if $T:\ch\rightarrow\ch$ is not bounded, since $T:\cg\rightarrow\cg^*$ is bounded, in general, it is easier to prove that $T\in C^1(A,\cg,\cg^*)$ than $T\in C^1(A)$.

If $\cg$ is the form domain of $H$, we have the following:
\begin{proposition}[see p. 258 of \cite{ABG}]
Let $A$ and $T$ be self-adjoint operator in a Hilbert space $\ch$. Assume that the unitary group $\{\exp(iA\tau)\}_{\tau\in\R}$ leaves the form domain $\cg$ of $T$ invariant. Then 
\begin{enumerate}
\item $T$ is of class $C^k(A)$ if $T\in C^k(A,\cg,\cg^*)$, for all $k\in\N$.

\item $T$ is of class $C^{1,1}(A)$ if $T\in C^{1,1}(A,\cg,\cg^*)$.
\end{enumerate}
\end{proposition}

As previously, since $T:\cg\rightarrow\cg^*$ is always bounded, it is, in general, easier to prove that $T\in C^k(A,\cg,\cg^*)$ than $T\in C^k(A)$.

Now we will recall the
\emph{H\"older-Zygmund continuity classes} of order
$s\in\,(0,\infty)\,$. Let $\ce$ be a Banach space and $F:\R\to\ce$ a
continuous function. If $0<s<1$ then $F$ is of class $\Lambda^s$ if
$F$ is H\"older continuous of order $s$.  If $s=1$ then $F$ is of
class $\Lambda^1$ if it is of Zygmund class, i.e.
$\|F(t+\varepsilon)+F(t-\varepsilon)-2F(t)\|\leq C\varepsilon$ for all
real $t$ and $\varepsilon>0$. If $s>1$, let us write $s=k+\sigma$ with
$k\geq1$ integer and $0<\sigma\leq1$; then $F$ is of class $\Lambda^s$
if $F$ is $k$ times continuously differentiable and $F^{(k)}$ is of
class $\Lambda^\sigma$. We said that $V\in \Lambda^s(A_u,\ch^1,\ch^{-1})$ if the function $\tau\mapsto V_\tau=e^{i\tau A_u}Ve^{-i\tau A_u}\in\rb(\ch^1,\ch^{-1})$ is of class $\Lambda^s$. Remark that, if $s\geq 1$ is an integer, $C^s(A_u,\ch^1,\ch^{-1})\subset \Lambda^s(A_u,\ch^1,\ch^{-1})$.

\section{Sub-exponential bounds on possible eigenvectors}\label{s: exp bounds}

In this section we will prove Theorem \ref{th:exp bound 1} and Theorem \ref{th:exp bound 2}. 

\subsection{The operator version}
Our proof of Theorem \ref{th:exp bound 1} closely follows the one of Theorem 2.1 in \cite{FH}. Therefore, we focus on the main changes. We will use notations of Theorem \ref{th:exp bound 1}.

For $\epsilon>0$ and $\tau>0$, define the real valued functions $F$ and $g$ by
\begin{equation}
F(x)=\tau \ln\biggl(\langle x\rangle(1+\epsilon\langle x\rangle)^{-1}\biggr)\text{ and }
\nabla F(x)=x g(x).\label{eq:premier F}
\end{equation}

Let $E\in\R$ and $\psi\in\cd(H)$ such that $H\psi=E\psi$. Let $\psi _F=\exp(F)\psi$. On the domain of $H$, we consider the operator
\begin{equation}
H(F)= e^F H e^{-F}= H-(\nabla F)^2+(ip\nabla F+i\nabla F p).\label{eq:H(F)}
\end{equation}

As in \cite{FH}, $\psi _F \in \cd (\Delta)=\cd(H(F))$,
\begin{equation}\label{eq:H(F)psiF}
H(F) \psi _F=E\psi _F 
\end{equation}
\begin{equation}\label{eq:H psiF}
\text{and }(\psi _F,H\psi _F)=(\psi _F,((\nabla F)^2+E)\psi _F).
\end{equation}

If we suppose in addition that 
\begin{equation}\label{eq:hyp borne exp}
\langle q\rangle^{\beta\tau} \exp(\alpha\langle q\rangle^\beta)\psi \in L^2(\R^\nu)
\end{equation}
 for all $\tau$ and some fixed $\alpha\geq 0$, $0<\beta<1$, then \eqref{eq:H(F)psiF} and \eqref{eq:H psiF} holds true for the new functions $F$ and $g$ given by
\begin{equation}
F(x)=\alpha\langle x\rangle^\beta+\tau \ln(1+\gamma\langle x\rangle^\beta\tau^{-1})\text{ and }
\nabla F(x)=x g(x)\label{eq:deuxieme F}
\end{equation}

for all $\gamma>0$ and $\tau>0$.

To replace Formula (2.9) in \cite{FH}, we prove the following
\begin{lemma}\label{l:Commutator psiF}
Suppose that $V$ is $\Delta$-compact. Let $u\in\cu$. Assume that $H=\Delta+V$ is of class $C^1(A_u)$. For both definitions of $F$ and $g$, we have
\begin{eqnarray}\label{eq:Commutator psiF}
(\psi _F,[H,iA_u]\psi _F)&=&(\psi_F,[(\nabla F)^2-q\cdot \nabla g,iA_u]\psi_F)\nonumber\\
& &-4\biggl\|\lambda(p)^{1/2}g^{1/2}A_D\psi_F\biggr\|^2-2\Re\biggl(gA_D\psi_F,i\nabla\lambda (p)\cdot p\psi_F\biggr)\nonumber\\
& &+4\Re\biggl([g^{1/2},\lambda (p)]g^{1/2}A_D\psi_F,A_D\psi_F\biggr).
\end{eqnarray}
\end{lemma}

We make some remarks about this Lemma:
\begin{enumerate}[(a)]

\item In the case \eqref{eq:premier F},
note that $\jap{x}g^{1/2}(x)$ is bounded. Thus $\biggl\|\lambda(p)^{1/2}g^{1/2}A_D\psi_F\biggr\|$ is well defined. 

\smallskip

\item In the case \eqref{eq:deuxieme F}, suppose that \eqref{eq:hyp borne exp} is true for all $\tau$ and some fixed $\alpha\geq 0$, $0<\beta<1$, we have
\begin{eqnarray*}
A_D\psi_F&=&p\cdot q\psi_F+\frac{i}{2}\psi_F\\
&=&p\cdot \frac{q}{\jap{q}}\jap{q}\biggl(1+\gamma\langle q\rangle^\beta\tau^{-1}\biggr)^{\tau} \exp(\alpha\langle q\rangle^\beta)\psi+\frac{i}{2}\psi_F.
\end{eqnarray*} 
Thus, $\psi_F\in L^2(\R^\nu)$ and \[\jap{q}\biggl(1+\gamma\langle q\rangle^\beta\tau^{-1}\biggr)^{\tau} \exp(\alpha\langle q\rangle^\beta)\psi\in L^2(\R^\nu).\] Moreover, since $H\psi=E\psi$ and $\nabla F$ is bounded for all $\tau>0$, we can show that  \[\jap{q}\biggl(1+\gamma\langle q\rangle^\beta\tau^{-1}\biggr)^{\tau} \exp(\alpha\langle q\rangle^\beta)\psi\in \ch^1.\]
Thus $\biggl\|\lambda(p)^{1/2}g^{1/2}A_D\psi_F\biggr\|$ is well defined.

\smallskip

\item If $V:\ch^1\rightarrow\ch^{-1}$ is compact and $V\in C^1(A_u,\ch^1,\ch^{-1})$, Lemma \ref{l:Commutator psiF} is still true with the same proof.
\end{enumerate}

\begin{proof}[Lemma \ref{l:Commutator psiF}]
 Since $V$ is of class $C^1(A_u,\cg,\cg^*)$ whith $\cg=\ch^2$ if $V$ is $\Delta$-compact, $\cg=\ch^1$ if $V:\ch^1\rightarrow\ch^{-1}$ is compact, by a simple computation, we can show that $e^F\Delta e^{-F}$ is of class $C^1(A_u,\cg,\cg^*)$, which implies that $H(F)=e^F\Delta e^{-F}+V$ is of class $C^1(A_u,\cg,\cg^*)$. For $\phi\in\cd(H)\cap\cd(A_u)$, we have
\begin{eqnarray*}
(\phi ,[H,iA_u]\phi )&=&(\phi ,[(H-H(F)),iA_u]\phi)+(\phi ,[H(F),iA_u]\phi )\\
&=&((H-H(F))\phi,iA_u\phi)-(A_u\phi,i(H-H(F))\phi)\\
& &+(\phi ,[H(F),iA_u]\phi )
\end{eqnarray*}

By using \eqref{eq:H(F)} and \eqref{eq:H(F)psiF}, we have:
\begin{equation*}
(H-H(F))\phi=((\nabla F)^2-(ip\nabla F+i\nabla F p))\phi
\end{equation*}

A simple computation gives
\begin{equation*}
(ip\nabla F+i\nabla F p)\phi=i(p(qg)+(qg)p)\phi=q\cdot \nabla g\phi+2ig A_D\phi
\end{equation*}
We have
\[(H-H(F))\phi=((\nabla F)^2-q\cdot \nabla g-2ig A_D)\phi\]
thus
\begin{multline}\label{eq:commut H}
(\phi ,[H,iA_u]\phi )=(\phi,[(\nabla F)^2-q\cdot \nabla g,iA_u]\phi)\\
-(2gA_D\phi,A_u\phi)-(A_u\phi,2gA_D\phi)+(\phi ,[H(F),iA_u]\phi )
\end{multline}

Since $u(x)=x\lambda(x)$,
\begin{equation}\nonumber
 A_u =\frac{1}{2}(\lambda(p)p\cdot q+q\cdot\lambda(p)p)
 =\lambda(p)A_D+\frac{1}{2}[q,\lambda(p)]p
\end{equation}

Using the Fourier transform, we see that $[q,\lambda(p)]=i\nabla\lambda(p)$.

Therefore
\[A_u=\lambda(p)A_D+\frac{i}{2}\nabla\lambda (p)\cdot p\]
which implies
\[
\begin{cases}
(2gA_D\phi,A_u\phi)=(2gA_D\phi,\lambda(p)A_D\phi)+(2gA_D\phi,\frac{i}{2}\nabla\lambda (p)\cdot p\phi)\\
(A_u\phi,2gA_D\phi)=(\lambda(p)A_D\phi,2gA_D\phi)+(\frac{i}{2}\nabla\lambda (p)\cdot p\phi,2gA_D\phi)
\end{cases}
\]

By sum, we get
\begin{eqnarray*}
(2gA_D\phi,A_u\phi)+(A_u\phi,2gA_D\phi)&=&2(A_D\phi,(g\lambda(p)+\lambda(p)g)A_D\phi)\\
& &+(gA_D\phi,i\nabla\lambda (p)\cdot p\phi)\\
& &+(i\nabla\lambda (p)\cdot p\phi,gA_D\phi).
\end{eqnarray*}
Since $g$ and $\lambda$ are positive,
\[
g\lambda(p)+\lambda(p)g=2g^{1/2}\lambda(p)^{1/2}\lambda(p)^{1/2}g^{1/2}
+g^{1/2}[g^{1/2},\lambda(p)]+[\lambda(p),g^{1/2}]g^{1/2}.
\]
This yields
\begin{eqnarray*}
(A_D\phi,(g\lambda(p)+\lambda(p)g)A_D\phi)&=&2\biggl\|\lambda(p)^{1/2}g^{1/2}A_D\phi\biggr\|^2\\
& &+2\Re\biggl(g^{1/2}A_D\phi,[g^{1/2},\lambda (p)]A_D\phi\biggr).
\end{eqnarray*}

So from \eqref{eq:commut H}, we obtain 
\begin{eqnarray}\label{eq:commut H 2}
(\phi,[H,iA_u]\phi )&=&(\phi,[(\nabla F)^2-q\cdot \nabla g,iA_u]\phi)\nonumber\\
& &-4\biggl\|\lambda(p)^{1/2}g^{1/2}A_D\phi\biggr\|^2-2\Re\biggl(gA_D\phi,i\nabla\lambda (p)\cdot p\phi\biggr)\nonumber\\
& &-4\Re\biggl(g^{1/2}A_D\phi,[g^{1/2},\lambda (p)]A_D\phi\biggr)+(\phi ,[H(F),iA_u]\phi ).
\end{eqnarray}
Remark that if $F$ satisfies \eqref{eq:premier F}, since $\jap{q}g^{1/2}$ is bounded, all operators which appears on the right hand side of \eqref{eq:commut H 2} are bounded in the $\ch^1$ norm. In particular, this equation can be extended to a similar equation for $\phi\in\cg\subset\ch^1$. Thus, since $\psi_F\in\cg$, we obtain a similar equation by replacing $\phi$ by $\psi_F$.

If $F$ satisfies \eqref{eq:deuxieme F}, we can see that all operators which appears on the right hand side of \eqref{eq:commut H 2} are bounded in the $\ch^1_1$ norm.  In particular, this equation can be extended to a similar equation for $\phi\in\ch^2_1$ if $V$ is $\Delta$-compact, $\phi\in\ch^1_1$ if $V:\ch^1\rightarrow\ch^{-1}$ is compact.

In all cases, since, by the Virial theorem and \eqref{eq:H(F)psiF}, $ (\psi_F ,[H(F),iA_u]\psi_F )=0$, we obtain \eqref{eq:Commutator psiF}.
%
\qed
\end{proof}

Since in \eqref{eq:Commutator psiF}, we do not know an explicit form for the commutator $[(\nabla F)^2-q\cdot \nabla g,iA_u]$, as in \cite{FH}, we need to control the size of this expression.

\begin{lemma}\label{l:H-S}
Let $f:\R^\nu\rightarrow \R$ be a $C^\infty$ application such that $f\in\cs^{\rho}$.\\
Then $\jap{q}^{-\rho}[f(q),iA_u]$ is bounded for all $C^\infty$ vector fields $u$ with bounded derivatives.
\end{lemma}

\begin{proof}
Suppose that $f\in\cs^{\rho}$. Then
\[
\forall k\in\N, \qquad \sup\limits_{t\in\R^\nu}\{\langle t\rangle^{-\rho+k}|\partial^\alpha_t f(t)|\}<\infty
\]
for all $\alpha$ multi-index such that $|\alpha|=k$.

Since $[f(q),q]=0$ and 
\[A_u=q\cdot u(p)-\frac{i}{2} div(u)(p),\]
 we have
\begin{equation}\label{eq:commutator}
[f(q),i A_u]= [f(q),i q\cdot u(p)+\frac{1}{2} div(u)(p)]=i q\cdot [f(q),u(p)]+\frac{1}{2} [f(q),div(u)(p)]
\end{equation}

By using the Helffer-Sj\"ostrand formula on $[f(q),u(p)]$, with $B=q$, $T=u(p)$ and\\ $\varphi(x)=f(x)$, we have:
\begin{equation}
 [f(q),u(p)]=i\nabla f(q)div(u)(p)+I_2
\end{equation}
where $I_2$ is the rest of the development of order $2$ in \eqref{eq: H-S formula}.
Similarly, 
\begin{equation}
 [f(q),div(u)(p)]=i\nabla f(q)\nabla div(u)(p)+I_2'
\end{equation}

So, from \eqref{eq:commutator}, we have
\begin{equation}\label{eq:H-S 1}
 [f(q),iA_u]=-q\cdot\nabla f(q)div(u)(p)-q\cdot I_2+\frac{i}{2}\nabla f(q)\nabla div(u)(p)+I_2'
\end{equation}
From Proposition \ref{prop: estimate H-S}, we deduce, since $f\in\cs^{\rho}$, that $\langle q\rangle^{s}I_2$ and $\langle q\rangle^{s}I_2'$ are bounded if $s<-\rho+2$. Moreover, since $f\in\cs^{\rho}$, $\langle x\rangle^{-\rho+1}\nabla f(x)$ is bounded, and we conclude that $\langle q\rangle^{-\rho}q\cdot\nabla f(q)$ is bounded. Since, by assumptions, $div(u)(p)$ and $\nabla div(u)(p)$ are bounded, by sum, $\langle q\rangle^{-\rho}[f(q),iA_u]$ is bounded.
\qed
\end{proof}

\begin{proof}[Theorem \ref{th:exp bound 1}]
Suppose that $E\notin\ce_u(H)$.

Let $F(x)=\tau \ln(\langle x\rangle(1+\epsilon \langle x\rangle)^{-1})$ and 
$ \Psi_\epsilon =\psi_F/\|\psi_F\|$.

Following \cite[equations (2.11) and (2.12)]{FH}, we can prove that $\nabla \Psi_\epsilon$ is bounded and that $(\Delta+1)\Psi_\epsilon$ converges weakly to zero as $\epsilon\rightarrow 0$. Thus, for all $\eta>0$, since $\langle q\rangle^{-\eta}(\Delta+1)^{-1}$ is compact, $\|\langle q\rangle^{-\eta}\Psi_\epsilon\|$ converges to 0 and, similarly, $\|\langle q\rangle^{-\eta}\nabla\Psi_\epsilon\|$ converges to 0.

From Lemma \ref{l:Commutator psiF}, we deduce that 
\begin{eqnarray}\label{eq:major commut psiF}
\biggl(\Psi _\epsilon,[H,iA_u]\Psi _\epsilon\biggr)&\leq& \biggl(\Psi_\epsilon,[(\nabla F)^2-q\cdot \nabla g,iA_u]\Psi_\epsilon\biggr)\nonumber\\
& &-2\Re\biggl(gA_D\psi_F,i\nabla\lambda (p)\cdot p\psi_F\biggr)\nonumber\\
& &-2\Re\biggl(gA_D\Psi_\epsilon,i\nabla\lambda (p)\cdot p\Psi_\epsilon\biggr).
\end{eqnarray}
Since $((\nabla F)^2-q\cdot \nabla g)$ is in $\cs^{-2}$, by Lemma \ref{l:H-S}, we have $\langle q\rangle^2[(\nabla F)^2-q\cdot \nabla g,iA_u]$ is bounded. Thus the first term on right side of \eqref{eq:major commut psiF} converges to zero as $\epsilon\rightarrow0$.
By assumptions, $\nabla\lambda (p)\cdot p$ is bounded. Since $\nabla \Psi_\epsilon$ is bounded, $\jap{q}^{-1}A_D\Psi_\epsilon$ is bounded and , for all $\mu>0$, $\|\jap{q}^{-1-\mu}A_D\Psi_\epsilon\|$ converges to zero as $\epsilon\rightarrow0$. Thus, since $\langle q\rangle^2 g$ is bounded, the last term on the right side of \eqref{eq:major commut psiF} converges to zero as $\epsilon\rightarrow0$.

Moreover, by the Helffer-Sjostrand formula, we have
\[[g^{1/2},\lambda (p)]=-i\nabla(g^{1/2})\nabla \lambda(p)+I\]
with $\jap{q}I\jap{q}^{s'}$ bounded for $s'<1$. In particular, $\jap{q}[g^{1/2},\lambda (p)]\jap{q}^{s'}$ is bounded for all $s'<1$. Thus,
\[
\left\|\jap{q}[g^{1/2},\lambda (p)]g^{1/2}A_D\Psi_\epsilon \right\|=\left\|\jap{q}[g^{1/2},\lambda (p)]g^{1/2}\jap{q}^{3/2}\jap{q}^{-3/2}A_D\Psi_\epsilon \right\|,\]
and since $\jap{q}g^{1/2}$ is bounded, the second term on the right side of \eqref{eq:major commut psiF} converges to zero as $\epsilon\rightarrow0$.

Thus, we deduce that
\[\limsup_{\epsilon\rightarrow 0}\biggl(\Psi _\epsilon,[H,iA_u]\Psi _\epsilon\biggr)\leq 0.\]

We follow \cite[equations (2.16) to (2.19)]{FH}to prove that, if $E\notin \ce_u(H)$, then
 \[\langle x\rangle^\tau \psi \in L^2(\R^\nu)\qquad \forall \tau>0.\]

\smallskip

Suppose now that the Theorem \ref{th:exp bound 1} is false so that 
\begin{equation}
S_E=\alpha_1 ^2+E
\end{equation}

where $\alpha_1>0$ and $S_E \notin \ce_u(H)$. By definition of $\ce_u(H)$, we have \eqref{eq:Mourre estimate} for some $\delta>0$, some $c_0>0$ and some compact operator $K_0$ with $I=[S_E -\delta,S_E +\delta]$.

As in \cite[equations (2.22) and (2.23)]{FH}, let $\alpha\in (0,\alpha_1)$ such that 
\[\alpha^2+E\in [S_E -\delta/2,S_E +\delta/2].\]
Let $0<\beta<1$. We have for all $\tau>0$
\begin{equation}
\langle x\rangle^{\beta\tau}\exp(\alpha \langle x\rangle^\beta)\psi \in L^2(\R^\nu).
\end{equation}

 Suppose $\gamma>0$ such that $\alpha+\gamma>\alpha_1$. So we have 

\begin{equation}\label{eq:exp}
\|\exp((\alpha+\gamma)\langle x\rangle^\beta)\psi\|=+\infty.
\end{equation}

In the following, we suppose that $\gamma$ is sufficiently small, $\gamma \in (0,1]$.
We denote by $b_j,\\  j=1,2,\cdots$ constants which are independant of $\alpha$, $\gamma$ and $\tau$.

Let $F(x)=\alpha\langle x\rangle^\beta+\tau \ln(1+\gamma\langle x\rangle^\beta\tau^{-1})$ and $\psi_F=\exp(F)\psi$, $\Psi_\tau=\psi_F/\|\psi_F\|$. 

By a simple estimate, we have $|x\nabla g(x)|\leq b_1 \langle x\rangle^{\beta-2}$ and \[(\nabla F)^2(x)\leq(\alpha+\gamma)^2\langle x\rangle^{2\beta-2} \leq(\alpha+\gamma)^2.\]

As previously, \eqref{eq:major commut psiF} is true.
Since $((\nabla F)^2-q\cdot \nabla g)$ is in $\cs^{2\beta-2}$, by Lemma \ref{l:H-S}, we have $\langle q\rangle^{2-2\beta}[(\nabla F)^2-q\cdot \nabla g,iA_u]$ is bounded. Therefore, the first term on right side of \eqref{eq:major commut psiF} converges to zero as $\tau\rightarrow\infty$.
By assumptions, $\nabla\lambda (p)\cdot p$ is bounded. As previously $\jap{q}^{-1}A_D\Psi_\tau$ is bounded and , for all $\mu>0$, $\|\jap{q}^{-1-\mu}A_D\Psi_\tau\|$ converges to zero as $\tau\rightarrow+\infty$. Thus, since $\langle q\rangle^{2-\beta} g$ is bounded, the last term on the right side of \eqref{eq:major commut psiF} converges to zero as $\tau\rightarrow+\infty$.

Moreover, by the Helffer-Sjostrand formula, we have
\[[g^{1/2},\lambda (p)]=-i\nabla(g^{1/2})\nabla \lambda(p)+I\]
with $\jap{q}^sI\jap{q}^{s'}$ bounded for $s<2$, $s'<1$ and $s+s'<3-\frac{\beta}{2}$.\\
In particular, $\jap{q}^{1}[g^{1/2},\lambda (p)]\jap{q}^{1/2}$ is bounded. Thus,
\[
\left\|\jap{q}[g^{1/2},\lambda (p)]g^{1/2}A_D\Psi_\epsilon \right\|=\left\|\jap{q}[g^{1/2},\lambda (p)]\jap{q}^{1/2}g^{1/2}\jap{q}^{-1/2}A_D\Psi_\epsilon \right\|,\]
and since $\jap{q}^{1-\frac{\beta}{2}}g^{1/2}$ is bounded, the second term on the right side of \eqref{eq:major commut psiF} converges to zero as $\tau\rightarrow+\infty$.

Thus, we deduce that 
\[\limsup_{\tau\rightarrow \infty}\biggl(\Psi _\tau,[H,iA_u]\Psi _\tau\biggr)\leq0.\]

As in \cite{FH}, we have 
\[\limsup\limits_{\tau\rightarrow+\infty}\left\|(H-E-(\nabla F)^2)\Psi_\tau\right\|=\limsup\limits_{\tau\rightarrow+\infty}\left\|(p\cdot\nabla F+\nabla F\cdot p)\Psi_\tau\right\|\]
and by a simple computation, we have
 \[p\cdot\nabla F+\nabla F\cdot p=2\nabla F\cdot p+i\Delta F\]
and we have
 \[
 \|(2\nabla F\cdot p+i\Delta F)\Psi_\tau\|\leq 2\|\nabla F\cdot \nabla \Psi_\tau\|+\|\Delta F\Psi_\tau\|.
 \]
 Since $|\nabla F|(x)\leq b_3\langle x\rangle^{\beta-1}$ and $|\Delta F|(x)\leq b_4 \langle x\rangle^{\beta-2}$, 
 \[\limsup\limits_{\tau\rightarrow+\infty}\|(H-E-(\nabla F)^2)\Psi_\tau\|=0\]
  which implies that 
  \[\limsup\limits_{\tau\rightarrow+\infty}\|(H-E-\alpha^2)\Psi_\tau\|\leq b_5 \gamma\]
By following \cite[equations (2.37) to (2.41)]{FH}, we deduce that

\begin{equation}
\underset{\tau \rightarrow \infty}{\liminf} (\Psi_\tau,E(I)[H,iA_u]E(I)\Psi_\tau)\geq c_0(1-(b_6\gamma)^2).\label{eq:new minor}
\end{equation}

Moreover, since
\begin{equation*}
\limsup\limits_{\tau \rightarrow \infty} (\Psi_\tau,[H,iA_u]\Psi_\tau)\leq 0,
\end{equation*}

we have

\begin{equation}\label{eq:new major}
\underset{\tau \rightarrow \infty}{\limsup} (\Psi_\tau,E(I)[H,iA_u]E(I)\Psi_\tau)\leq b_7 \gamma.
\end{equation}

From \eqref{eq:new minor} and \eqref{eq:new major}, we have
\[c_0(1-(b_6\gamma)^2)\leq b_7 \gamma.\]
Since $c_0$ is a fixed positive number, we have a contradiction for all small enough $\gamma>0$. Thus the theorem is proved.
\qed
\end{proof}

\subsection{The form version}

If we only suppose that $V:\ch^1\rightarrow\ch^{-1}$ is bounded with bound less than one, we have the following

\begin{proof}[Theorem \ref{th:exp bound 2}]
Suppose that $E\notin\ce_u(H)$. We denote $C_i>0$ constant independant of $\epsilon$.

Let $F(x)=\tau \ln(\langle x\rangle(1+\epsilon \langle x\rangle)^{-1})$ and 
$ \Psi_\epsilon =\psi_F/\|\psi_F\|$.
As in \cite{FH}, we can prove that for any bounded set $B$
\[\lim\limits_{\epsilon\rightarrow 0}\int_B|\Psi_\epsilon|^2d^n x=0.\]
By a simple calculus, we have
\[\nabla\psi_F=\nabla F\psi_F+e^F\nabla\psi.\]
So, for any bounded set $B$, since $\nabla F$ and $e^F$ are uniformly bounded in $\epsilon$ on $B$, we have
\begin{eqnarray*}
\left(\int_B|\nabla\Psi_\epsilon|^2d^nx\right)^{1/2}&\leq& \left(\int_B|\nabla F\Psi_\epsilon|^2d^nx\right)^{1/2}+\left(\int_B|e^F\nabla\psi|^2d^nx\right)^{1/2}\|\psi_F\|^{-1}\\
&\leq& C_1\left(\int_B|\Psi_\epsilon|^2d^nx\right)^{1/2}+C_2\left(\int_B|\nabla\psi|^2d^nx\right)^{1/2}\|\psi_F\|^{-1}\\
&\leq& C_1\left(\int_B|\Psi_\epsilon|^2d^nx\right)^{1/2}+C_2\left(\int_{\R^\nu}|\nabla\psi|^2d^nx\right)^{1/2}\|\psi_F\|^{-1}.
\end{eqnarray*}
 Since $V:\ch^1\rightarrow\ch^{-1}$ is bounded with bound less than one, we have 
 \[H+1=\jap{p}(1+\jap{p}^{-1}V\jap{p}^{-1})\jap{p}\]
 which implies that 
\[\jap{p}\psi=(1+\jap{p}^{-1}V\jap{p}^{-1})^{-1}\jap{p}^{-1}(H+1)\psi=(E+1)(1+\jap{p}^{-1}V\jap{p}^{-1})^{-1}\jap{p}^{-1}\psi\]
and so, 
\[\|\nabla \psi\|=\|\frac{p}{\jap{p}}\jap{p}\psi\|\leq C_3\|\psi\|.\]
All of this implies that 
\[\lim\limits_{\epsilon\rightarrow 0}\int_B|\nabla\Psi_\epsilon|^2d^n x=0.\]
Moreover, since $V:\ch^1\rightarrow\ch^{-1}$ is bounded with bound less than one, there is $0<a<1$ and $0<b$ such that
\[\left|(\Psi_\epsilon,V\Psi_\epsilon)\right|\leq a\|\nabla\Psi_\epsilon\|^2+b.\]
So, by \eqref{eq:H psiF}, we have
\[(1-a)\|\nabla\Psi_\epsilon\|^2-b\leq (\Psi_\epsilon,H\Psi_\epsilon)\leq (E+\tau^2).\]
So $\|\nabla\Psi_\epsilon\|$ is bounded as $\epsilon\rightarrow 0$ and with a similar argument, $\|\jap{p}\Psi_\epsilon\|$ is bounded as $\epsilon\rightarrow 0$
So, for all $N>0$, if $\chi_N$ is the characteristic function of $\{x:\jap{x}\leq N\}$, we have
\begin{eqnarray*}
\|\nabla F\nabla\Psi_\epsilon\|&\leq& \|\chi_N\nabla F\nabla\Psi_\epsilon\|+\|(1-\chi_N)\nabla F\nabla\Psi_\epsilon\|\\
&\leq& C_4\|\chi_N\nabla\Psi_\epsilon\|+\tau N^{-1}\|\nabla\Psi_\epsilon\|.
\end{eqnarray*}
Since this inequality is true for all $N>0$ and $\|\nabla\Psi_\epsilon\|$ is bounded as $\epsilon\rightarrow 0$, 
\[\lim\limits_{\epsilon\rightarrow 0}\|\nabla F\nabla\Psi_\epsilon\|=0\]
and as in \cite[equation (2.13)]{FH}, we deduce that 
\[\lim\limits_{\epsilon\rightarrow 0}\|(H-E)\Psi_\epsilon\|=0\]
 which implies 
 \[\begin{cases}
 \lim\limits_{\epsilon\rightarrow 0}\|E(\R\backslash I)\Psi_\epsilon\|=0\\
 \lim\limits_{\epsilon\rightarrow 0}\|(H+i)E(\R\backslash I)\Psi_\epsilon\|=0
 \end{cases}.\]
 As previously, by writing $\jap{p}=(1+\jap{p}^{-1}V\jap{p}^{-1})^{-1}\jap{p}^{-1}(H+1)$, we deduce that 
 \[\lim\limits_{\epsilon\rightarrow0}\|\jap{p}E(\R\backslash I)\Psi_\epsilon\|=0.\]
So if $f_1(\epsilon)=(\Psi_\epsilon, E(\R\backslash I)[H,iA_u]\Psi_\epsilon)$, we have
\begin{eqnarray*}
\lim\limits_{\epsilon\rightarrow0}|f_1(\epsilon)|&\leq& \lim\limits_{\epsilon\rightarrow0}\|\jap{p}E(\R\backslash I)\Psi_\epsilon\|\cdot\|\jap{p}^{-1}[H,iA_u]\jap{p}^{-1}\|\cdot \|\jap{p}\Psi_\epsilon\|\\
&=&0
\end{eqnarray*} 
and simlarly with $f_2(\epsilon)=(\Psi_\epsilon, E(I)[H,iA_u]E(\R\backslash I)\Psi_\epsilon)$.
Remark that we can prove similar things with $F(x)=\alpha\langle x\rangle^\beta+\tau \ln(1+\gamma\langle x\rangle\tau^{-1})$. Thus, by using a similar proof than for Theorem \ref{th:exp bound 1}, Theorem \ref{th:exp bound 2} is proved.
\qed
\end{proof}

\section{Possible eigenvectors can not satisfies sub-exponential bounds}\label{s:abs vp}

In this section, we will prove Theorem \ref{th:abs 1} and Theorem \ref{th:abs 2}.

\begin{proof}[Theorem \ref{th:abs 1}]
In this proof, we will follow the method used in \cite[Theorem 4.18]{CFKS}.

Suppose that Theorem \ref{th:abs 1} is false: there is $\psi \not=0$ such that 
\[\exp(\alpha\langle x\rangle^\beta)\psi\in L^2(\R^\nu)\]
 for all $\alpha>0$, $0<\beta<1$ and $H\psi=E\psi$ with $E>0$.
 For $\alpha>0$, $0<\beta<1$, let $F_\beta(x)=\alpha \langle x\rangle^\beta$. As previously, we denote $\psi_F=\exp(F_\beta(q))\psi$ and $xg_\beta(x)=\nabla F_\beta(x)$. By direct calculation, we have $\nabla F_\beta(x)=\alpha\beta x\langle x\rangle^{\beta-2}$ and
\begin{equation}\label{eq:nablaF}\begin{cases}
|\nabla F_\beta|^2=\alpha^2\beta^2\langle x\rangle^{2\beta-2}(1-\langle x\rangle^{-2})\\
x\nabla(\nabla F_\beta(x))^2=2\alpha^2\beta^2\langle x\rangle^{2\beta-2}(1-\langle x\rangle^{-2})\left(\beta-1+(2-\beta)\langle x\rangle^{-2}\right)\end{cases}.
\end{equation}
 
 By assumptions, $\psi_F\in L^2(\R^\nu)$ for all $\alpha>0$, $0<\beta<1$. Suppose that there is $\delta>-2$, $\delta',\sigma,\sigma'\in\R$ such that $\delta+\delta'>-2$ and \eqref{eq:V,iA petit 1} is true.

Take $\alpha>0$ and $0<\beta<1$. We denote $C$ (possibly different) constants that do not depend on $\alpha$ or $\beta$.
From \eqref{eq:V,iA petit 1}, we derive
\begin{equation}\label{eq:premier minor}
(\psi_F,[H,iA_D]\psi_F)\geq(2+\delta)(\psi_F,\Delta\psi_F)+\delta' (\psi_F,(\nabla F_\beta)^2\psi_F)+(\sigma\alpha+\sigma')\|\psi_F\|^2.
\end{equation}
Since $V$ is $\Delta$-compact, we can find (see \cite{Ka2}), for all $0<\mu<1$, some $C_\mu>0$ such that
\[(\psi_F,\Delta\psi_F)\geq\mu(\psi_F,H\psi_F)-C_\mu\|\psi_F\|^2.\]
Inserting this information in \eqref{eq:premier minor} and using \eqref{eq:H psiF}, we get, for all $0<\mu<1$,
\begin{equation}\label{eq:minor abs}
(\psi_F,[H,iA_D]\psi_F)\geq\left((2+\delta)\mu+\delta'\right) (\psi_F,(\nabla F_\beta)^2\psi_F)+(\sigma\alpha+C-C_\mu)\|\psi_F\|^2.
\end{equation}

By \eqref{eq:Commutator psiF} with $\lambda(x)\equiv1$, we have
\[(\psi_F,[H,iA_D]\psi_F)\leq (\psi_F,((x\nabla)^2g_\beta-x\nabla(\nabla F_\beta(x))^2)\psi_F).\]
Since $|(x\nabla)^2g_\beta|\leq C\alpha$,
\[(\psi_F,[H,iA_D]\psi_F)\leq \alpha C\|\psi_F\|^2-(\psi_F,x\nabla(\nabla F_\beta(x))^2\psi_F).\]
Using \eqref{eq:nablaF} and the fact that
\[2\alpha^2\beta^2(2-\beta)(\psi_F,\langle q\rangle^{2\beta-4}(1-\langle q\rangle^{-2})\psi_F)\geq 0,\] we obtain
\begin{equation}\label{eq:major abs}
(\psi_F,[H,iA_D]\psi_F)\leq \alpha C\|\psi_F\|^2
-2\alpha^2\beta^2(\beta-1)(\psi_F,\langle q\rangle^{2\beta-2}(1-\langle q\rangle^{-2})\psi_F).
\end{equation}
Therefore, if we denote $\Psi_\alpha=\psi_F/\|\psi_F\|$, it follows from \eqref{eq:minor abs} and \eqref{eq:major abs} that
\begin{equation}\label{eq:contradiction}
\alpha^2\beta^2\left(\mu(2+\delta)+\delta'+2\beta-2 \right)(\Psi_\alpha,\langle q\rangle^{2\beta-2}(1-\langle q\rangle^{-2})\Psi_\alpha)\leq \alpha C+C.
\end{equation}
Since $2+\delta+\delta'>0$, we can choose $0<\mu<1$ such that $(2+\delta)\mu+\delta'>0$.
Taking $\beta-1$ small enough, we can ensure that
\begin{equation}\label{eq: choix beta}
\tau=\beta^2\left(\mu(2+\delta)+\delta'+2\beta-2 \right)>0.
\end{equation}
Remark that we can suppose that $\beta\geq 1/2$.
Since $t^{\beta-1}\exp(t^\beta)\geq 1$ for all $t\geq1$, we derive from \eqref{eq:contradiction} that, for $\alpha\geq1$,
\begin{equation}\label{eq:contradiction2}
(\alpha+1)C\geq \alpha^2\tau(\Psi_{\alpha-1},(1-\jap{q}^{-2})\Psi_{\alpha-1}).
\end{equation}
Since $\psi\not=0$, we can find $\epsilon>0$ such that $\|\mathbb{1}_{|\cdot|\geq2\epsilon}(q)\psi\|>0$.
For all $\alpha>0$,
\begin{eqnarray*}
\frac{\|\mathbb{1}_{|\cdot|\leq\epsilon}(q)\exp(\alpha\langle q\rangle^\beta)\psi\|^2}{\|\exp(\alpha\langle q\rangle^\beta)\psi\|^2}&\leq&\frac{\exp(2\alpha\langle \epsilon\rangle^\beta)\|\mathbb{1}_{|\cdot|\leq\epsilon}(q)\psi\|^2}{\exp(2\alpha\langle 2\epsilon\rangle^\beta)\|\mathbb{1}_{|\cdot|\geq2\epsilon}(q)\psi\|^2}\\
&\leq&\exp\left(2\alpha(\langle \epsilon\rangle^\beta-\langle 2\epsilon\rangle^\beta)\right)\frac{\|\psi\|^2}{\|\mathbb{1}_{|\cdot|\geq2\epsilon}(q)\psi\|^2}
\end{eqnarray*}
and
\begin{eqnarray*}
(\Psi_\alpha,(1-\langle q\rangle^{-2})\Psi_\alpha)&\geq&(1-\langle \epsilon\rangle^{-2})\frac{\|\mathbb{1}_{|\cdot|\geq\epsilon}(q)\psi_F\|^2}{\|\psi_F\|^2}\\
&\geq&(1-\langle \epsilon\rangle^{-2})\left(1-\frac{\|\mathbb{1}_{|\cdot|\leq\epsilon}(q)\psi_F\|^2}{\|\psi_F\|^2}\right)\\
&\geq&(1-\langle \epsilon\rangle^{-2})\biggl(1-C_\epsilon\exp\left(2\alpha(\langle \epsilon\rangle^\beta-\langle 2\epsilon\rangle^\beta)\right)\biggr)
\end{eqnarray*}
where $C_\epsilon=\frac{\|\psi\|^2}{\|\mathbb{1}_{|\cdot|\geq2\epsilon}(q)\psi\|^2}$.
So there exist $C_1>0$ and $\alpha_0>0$ such that for all $\alpha\geq\alpha_0$, 
\[(\Psi_\alpha,(1-\langle q\rangle^{-2})\Psi_\alpha)\geq C_1.\]
This implies, together with \eqref{eq:contradiction2} that, for $\alpha\geq\alpha_0$,
\[(\alpha+1)C\geq \alpha^2\tau C_1\]
which is false for $\alpha$ large enough.
\qed\end{proof}


\begin{proof}[Theorem \ref{th:abs 2}]
Suppose that $V:\ch^1\rightarrow\ch^{-1}$ is bounded and that Theorem \ref{th:abs 2} is false. We have:
\begin{eqnarray*}
(\psi_F,H\psi_F)&=&(\psi_F,\Delta\psi_F)+(\psi_F,V\psi_F)\\
&\leq&(\psi_F,\Delta\psi_F)+\|\jap{p}^{-1}V\jap{p}^{-1}\|\|\jap{p}\psi_F\|^2\\
&\leq&(\psi_F,\Delta\psi_F)+\|\jap{p}^{-1}V\jap{p}^{-1}\|\biggl((\psi_F,\Delta\psi_F)+\|\psi_F\|^2\biggr),
\end{eqnarray*}
which implies 
\[(\psi_F,\Delta\psi_F)\geq v(\psi_F,H\psi_F)-v\|\jap{p}^{-1}V\jap{p}^{-1}\|\|\psi_F\|^2,\]
where $v=(1+\|\jap{p}^{-1}V\jap{p}^{-1}\|)^{-1}$.
Using \eqref{eq:V,iA petit 2}, we obtain
\begin{eqnarray}
(\psi_F,[H,iA_D]\psi_F)&\geq&(2+\delta)(\psi_F,\Delta\psi_F)+\delta' (\psi_F,(\nabla F_\beta)^2\psi_F)+(\sigma\alpha+\sigma')\|\psi_F\|^2\nonumber\\
&\geq&(2+\delta)v(\psi_F,H\psi_F)+\delta' (\psi_F,(\nabla F_\beta)^2\psi_F)\nonumber\\
& &+(\sigma\alpha+C)\|\psi_F\|^2\nonumber\\
&\geq&\left((2+\delta)v+\delta'\right)(\psi_F,|\nabla F_\beta|^2\psi_F)+(\sigma\alpha+C)\|\psi_F\|^2.
\end{eqnarray}
By assumptions, $(2+\delta)v+\delta'>0$. Thus, we can choose $0<\beta<1$ such that
\[\tau=\beta^2\left((2+\delta)v+\delta'+2\beta-2 \right)>0.\]
Following the last lines of the proof of Theorem \ref{th:abs 1}, we get a contradiction for $\alpha$ large enough.
\qed
\end{proof}
\section{Concrete potentials}\label{s:examples}
In this section, we study the concrete potentials that we mentioned in the several remarks following our results in Section \ref{s:Intro}.
\subsection{Preliminary results}
We want to apply Theorem \ref{th:abs 1} and Theorem \ref{th:abs 2} to this concrete potentials. We thus have to check the validity of \eqref{eq:V,iA petit 1} and \eqref{eq:V,iA petit 2} for them. To this end, we shall need the following
\begin{lemma}\label{l:deriv sr}
Let $W$ be a bounded real valued function such that $|q|W$ is bounded ($W$ is of short range type for example) and the distributionnal $\nabla W$ is locally in $L^\infty$.
Let $V=|q|^{-1}q\cdot\nabla W+B+V_L$ with $B$ a bounded real valued function such that $qB$ is bounded and $V_L$ a bounded real valued function such that there is $\theta>0$ with $\jap{q}^\theta V_L$ and $q\nabla V_L$ are bounded ($V_L$ is a long range potential). Let $\psi\in L^2$ such that $H\psi=E\psi$ with $E>0$.
 For $\alpha>0$, $0<\beta<1$, let $F_\beta(x)=\alpha \langle x\rangle^\beta$ and $\psi_F=e^F\psi$. As in Theorem \ref{th:abs 2}, suppose that $\psi_F\in L^2$ for all $\alpha>0$, $0<\beta<1$.
Then, for all $\epsilon>0$, there is $C_\epsilon\in\R$, independent of $\alpha,\beta$, such that
\begin{multline}\label{eq:deriv sr}
2\Re(qV\psi_F,\nabla\psi_F)\geq-\biggl(3\epsilon+4\||q|W\|\biggr)\|\nabla\psi_F\|^2
-4\||q|W\|\cdot\|\nabla F\psi_F\|^2\\-C_\epsilon(\alpha+1)\|\psi_F\|^2.
\end{multline}
\end{lemma}

\begin{proof}[Lemma \ref{l:deriv sr}]
To begin, remark that since $|q|W$ is bounded, $W$ vanishes at infinity. Thus, by writing $\nabla W=[p,i W]$, we can show that $|q|^{-1}q \cdot\nabla W:\ch^1\rightarrow\ch^{-1}$ is compact and, by sum, that $V:\ch^1\rightarrow\ch^{-1}$ is compact.

 Let $\epsilon>0$. To simplify notations, we denote by $C,D_\epsilon$ possibly different constants independent of $\alpha,\beta$, where $D_\epsilon$ may depends on $\epsilon$. As in the proof of Theorem \ref{th:exp bound 2}, we can show that $\nabla\psi_F\in L^2$ for all $\alpha>0, 0<\beta<1$.
Recall that for $a,b\in\ch$, $\eta>0$, we have
\begin{eqnarray}\label{eq:major prod scal}
2|(a,b)|&\leq&2\|a\|\cdot\|b\|\nonumber\\
&\leq&\eta\|a\|^2+\eta^{-1}\|b\|^2.
\end{eqnarray}

Since $|q|V=q\cdot\nabla W+|q|B+|q|V_L$, we can write $qV=|q|\nabla W+qB+qV_L$.
Let $\kappa\in C^\infty_c(\R,\R)$ such that $\kappa(t)=1$ if $|t|<1$, $0\leq\kappa\leq1$.
\begin{eqnarray*}
\left|2\Re\biggl(qV\psi_F,\nabla\psi_F\biggr)\right|&\leq&2\left|\Re\biggl((\kappa(|q|)+(1-\kappa(|q|)))|q|\nabla W \psi_F,\nabla\psi_F\biggr)\right|\\
& &+2\left|\Re\biggl(q B+q V_L\psi_F,\nabla\psi_F\biggr)\right|,
\end{eqnarray*}
and using $\nabla W=[p,iW]$,
\begin{eqnarray}\label{eq:major lemma 5.1}
& &\left|2\Re\biggl(qV\psi_F,\nabla\psi_F\biggr)\right|\nonumber\\
&\leq&2\left|\biggl((1-\kappa(|q|))|q|W \nabla\psi_F,\nabla\psi_F\biggr)\right|+2\left|\biggl((1-\kappa(|q|))|q|W\psi_F,\Delta\psi_F\biggr)\right|\nonumber\\
& &+2\left|\biggl((\kappa(|q|)|q|\nabla W+qB+[(1-\kappa(|q|))|q|,p]W)\psi_F,\nabla\psi_F\biggr)\right|\nonumber\\
& &+\left|2\Re\biggl(qV_L\psi_F,\nabla\psi_F\biggr)\right|.
\end{eqnarray}
Since $\kappa(|q|)|q|\nabla W+qB+[(1-\kappa(|q|))|q|,p]W$ is bounded, by \eqref{eq:major prod scal}, we can see that the third term on the r.h.s. is less or equal to a term of the form $\epsilon\|\nabla\psi_F\|^2+D_\epsilon\|\psi_F\|^2$. 

For the last term on the r.h.s., remark that
\begin{eqnarray*}
\left|2\Re\biggl(qV_L\psi_F,\nabla\psi_F\biggr)\right|&\leq&\left|\biggl(\psi_F,q\cdot\nabla V_L\psi_F\biggr)\right|+\nu\left|\biggl(\psi_F,V_L\psi_F\biggr)\right|.
\end{eqnarray*}
Thus, since $V_L$ and $q\nabla V_L$ are bounded, there is $C>0$ such that the last term on the r.h.s. of \eqref{eq:major lemma 5.1} is less or equal to $C\|\psi_F\|^2$.

Since $0\leq\kappa\leq1$, we can remark that the first term on the r.h.s. of \eqref{eq:major lemma 5.1} is less or equal to $2\||q|W\|\cdot\|\nabla\psi_F\|^2$.

By \eqref{eq:H(F)} and \eqref{eq:H(F)psiF}, we can write
\begin{eqnarray*}
\Delta\psi_F&=&H\psi_F-V\psi_F\\
&=&(\nabla F)^2\psi_F-(ip\nabla F+i\nabla Fp)\psi_F+E\psi_F-V\psi_F\\
&=&(\nabla F)^2\psi_F-2\nabla F\nabla\psi_F-\Delta F\psi_F+E\psi_F-V\psi_F.
\end{eqnarray*}
Inserting this information in the second term on the r.h.s. of \eqref{eq:major lemma 5.1}, we get
\begin{eqnarray}\label{eq:major last term}
& &2\left|((1-\kappa(|q|))|q|W\psi_F,\Delta\psi_F)\right|\nonumber\\
&\leq&2\|(1-\kappa(|q|))|q|W\|\|\nabla F\psi_F\|^2+4\left|((1-\kappa(|q|))|q|W\nabla F\psi_F,p\psi_F)\right|\nonumber\\
& &+2\left|((1-\kappa(|q|))|q|W\psi_F,V\psi_F)\right|\nonumber\\
& &+2\left|((1-\kappa(|q|))|q|W\psi_F,\Delta F\psi_F)\right|+2\|(1-\kappa(|q|))|q|W\||E|\|\psi_F\|^2.
\end{eqnarray}
By \eqref{eq:major prod scal} with $\eta=1$, we can remark that the second term on the r.h.s. of \eqref{eq:major last term} is bounded above by $2\||q|W\|\|\nabla F\psi_F\|^2+2\||q|W\|\|\nabla\psi_F\|^2$. Since $\alpha^{-1}\Delta F$ is bounded, the 2 last terms on the r.h.s. of \eqref{eq:major last term} are less or equal to $C(\alpha+1)\|\psi_F\|^2$.

For the third term, we use that $B$, $V_L$ and $|q|W$ are bounded to arrive at
\begin{eqnarray*}
& &2\left|((1-\kappa(|q|))|q|W\psi_F,V\psi_F)\right|\\
&=&2\left|((1-\kappa(|q|))W\psi_F,(q\cdot\nabla W+|q|B+|q|V_L)\psi_F)\right|\\
&\leq&2\left|((1-\kappa(|q|))q\psi_F, W\nabla W\psi_F)\right|+C\|\psi_F\|^2\\
&\leq&2\left|((1-\kappa(|q|))q\psi_F, [p,iW^2]\psi_F)\right|+C\|\psi_F\|^2.
\end{eqnarray*}
Since $qW^2$ and $W$ are bounded, by \eqref{eq:major prod scal},
\begin{eqnarray*}
\left|((1-\kappa(|q|))q\psi_F,[p,iW^2]\psi_F)\right|&\leq&\left|(p(1-\kappa(|q|))q\psi_F,W^2\psi_F)\right|\\
& &+\left|((1-\kappa(|q|))qW^2\psi_F,p\psi_F)\right|\\
&\leq&2\left|((1-\kappa(|q|))qW^2\psi_F,p\psi_F)\right|\\
& &+\left|(W^2[p,(1-\kappa(|q|))q]\psi_F,\psi_F)\right|\\
&\leq&C\|\psi_F\|^2+\epsilon\|\nabla\psi_F\|^2.
\end{eqnarray*}
Thanks to these inequalities, we derive from \eqref{eq:major lemma 5.1}
\begin{eqnarray*}
\biggl|2\Re(|q|\nabla W\psi_F,\nabla\psi_F)\biggr|&\leq&(3\epsilon+4\||q|W\|)\|\nabla\psi_F\|^2\\
& &+C_\epsilon(\alpha+1)\|\psi_F\|^2+4\||q|W\|\|\nabla F\psi_F\|^2.
\end{eqnarray*}
which implies \eqref{eq:deriv sr}.
\qed\end{proof}
Remark that, in \eqref{eq:deriv sr}, we can replace $\||q|W\|$ by $\|((1-\kappa(|q|))|q|W\|$. In particular, if $|q|W$ vanishes at infinity,  we can choose the function $\kappa$ such that $\|((1-\kappa(|q|))|q|W\|\leq\epsilon$.

\begin{corollary}\label{c:deriv sr}
Let $W$ be a bounded real valued function such that $|q|W$ is bounded ($W$ is of short range type for example) and the distributionnal $\nabla W$ is locally in $L^\infty$.
Let $V=|q|^{-1}q\cdot\nabla W+B+V_L$ with $B$ a bounded real valued function such that $qB$ is bounded and $V_L$ a bounded real valued function such that $q\nabla V_L$ is bounded. If $\||q|W\|$ is small enough, we can choose $\epsilon>0$ small enough such that $V$ satisfies \eqref{eq:V,iA petit 1} and \eqref{eq:V,iA petit 2}.
\end{corollary}
Remark that, if we denote $g$ the function such that $xg(x)=\nabla F_\beta(x)$, the first term on the r.h.s. of \eqref{eq:major lemma 5.1} is less or equal to $\frac{C_1}{\alpha\beta}\left\|g^{1/2}A_D\psi_F\right\|+C_2\|\psi_F\|^2$ where $C_1,C_2$ are independent of $\alpha,\beta$. In particular, if $\alpha\beta$ is large enough, this term appears in \eqref{eq:V,iA petit 1'} and \eqref{eq:V,iA petit 2'}, and we can use these assumptions instead of \eqref{eq:V,iA petit 1} and \eqref{eq:V,iA petit 2}.

\begin{proof}[Corollary \ref{c:deriv sr}]
Let $\psi$ and $\psi_F$ as in Lemma \ref{l:deriv sr}. Then
\begin{eqnarray*}
(\psi_F,[V,iA_D]\psi_F)&=&(\psi_F,iVq\cdot p\psi_F)-(\psi_F,iq\cdot pV\psi_F)\\
&=&(qV\psi_F, \nabla\psi_F)-( \nabla\psi_F,qV\psi_F)-\nu(\psi_F,V\psi_F)\\
&=&2\Re(qV\psi_F, \nabla\psi_F)-\nu(\psi_F,V\psi_F).
\end{eqnarray*}
Let $\kappa\in C^\infty_c(\R,\R)$ such that $\kappa(t)=1$ if $|t|<1$, $0\leq\kappa\leq1$.
For second term, we have
\begin{eqnarray*}
\nu(\psi_F,V\psi_F)&=&\nu(\psi_F,\kappa(|q|)V\psi_F)+\nu(\psi_F,(1-\kappa(|q|))V\psi_F)\\
&=&\nu(\psi_F,(1-\kappa(|q|))(|q|^{-1}q\cdot\nabla W+B+V_L)\psi_F)+\\
& &\nu(\psi_F,\kappa(|q|)V\psi_F).
\end{eqnarray*}
By writing $\nabla W=[p,iW]$, since $B, \kappa(|q|)V, V_L$ and \\$[(1-\kappa(|q|))|q|^{-1}q,p]$ are bounded, for all $\epsilon>0$, by \eqref{eq:major prod scal},
\[\left|\nu(\psi_F,V\psi_F)\right|\leq\epsilon\|\nabla\psi_F\|^2+C_\epsilon\|\psi_F\|^2.\]
Using this and Lemma \ref{l:deriv sr}, we obtain \eqref{eq:V,iA petit 1} and/or \eqref{eq:V,iA petit 2} if $\||q|W\|$ and $\epsilon$ are small enough. \qed
\end{proof}

\subsection{A class of oscillating potential}\label{ss:oscillant}

Let $v\in C^1(\R^n,\R^n)$ with bounded derivative. Let $V_{lr},V_{sr},V_c,\tilde{V}_{sr}$ such that $V_c$ is compactly support and such that there is $\rho_{lr},\rho_{sr}>0$ with $\jap{q}^{1+\rho_{sr}}V_{sr}$, $\jap{q}^{1+\rho_{sr}}\tilde{V}_{sr}$, $\jap{q}^{\rho_{lr}}V_{lr}$, $\jap{q}^{\rho_{lr}}q\cdot\nabla V_{lr}$ are bounded ($V_{lr}$ is a long-range potential and $V_{sr}$ and $\tilde{V}_{sr}$ are  short-range potentials). Moreover, we suppose that $V_c:\ch^1\rightarrow\ch^{-1}$ is compact and that there is $\epsilon_c>0$ and $\lambda_c\in\R$ such that, for all $\phi\in\cd(H)\cap\cd(A_D)$, 
\[(\phi,[V_c,iA_D]\phi)\geq (\epsilon_c-2)(\phi,\Delta\phi)+\lambda_c\|\phi\|^2.\]

Let $\zeta,\theta\in\R$, $k>0$, $w\in\R^*$ and $\kappa\in C^\infty_c(\R,\R)$ such that $\kappa=1$ on $[-1,1]$ and $0\leq \kappa\leq 1$.
Let \begin{equation}\label{eq:oscillant}
W_{\zeta \theta}(x)=w(1-\kappa(|x|))\frac{\sin(k|x|^\zeta)}{| x|^\theta}.
\end{equation}
Remark that if we take $\zeta=\theta=1$, this potential has the form of the Wigner-von Neuman potential for which we know that $\frac{k^2}{4}$ is an eigenvalue. As pointed out in \cite{JMb}, Corollary \ref{c:FH} applies with $V_{lr}+V_{sr}+W_{\zeta \theta}$ as potential if $\theta>0$ and $\theta>\zeta$ or if $\theta>1$.  In \cite{JMb}, it is claimed that Corollary \ref{c:FH} applies when $1/2\geq\theta>0$, $\zeta>1$, $\zeta+\theta>2$ and $|w|$ small enough. The corresponding proof, however, is not sufficient. Here, thanks to our main result, we are able to prove the following
 
 \begin{proposition}\label{prop:oscillant}
 Let $V=V_{lr}+V_{sr}+v\cdot\nabla \tilde{V}_{sr}+V_c+W_{\zeta \theta}$ and let $H=\Delta+V an$.
 \begin{enumerate}
\item If $\zeta+\theta>1$, then $V:\ch^1\rightarrow\ch^{-1}$ is compact.
 
 \smallskip
 
\item If $\zeta+\theta\geq3/2$, then a possible eigenvector of $H$ has sub-exponential bounds.

\smallskip

\item If $\zeta+\theta>3/2$, then $A_u$ is conjugate to $H$ on all compact subset of $(0,+\infty)$ for all $u$ bounded. In particular the sub-exponential bounds are unlimited.
 
 \smallskip
 
 \item Let $\theta\in\R$, $\zeta>1$ and $\zeta+\theta=2$. If $\left|\frac{w}{k\zeta}\right|<\frac{\epsilon_c}{6}$, then $H=\Delta+W_{\zeta \theta}$ has no positive eigenvalue;
 
 \smallskip
 
 \item If $\zeta+\theta>2$, then $H=\Delta+W_{\zeta \theta}$ has no positive eigenvalue.
 \end{enumerate}
 \end{proposition}
 
 We will give some comments about this Proposition
 \begin{enumerate}
\item In the case $\zeta+\theta=2$, $\theta\leq 1/2$ and $\zeta>1$, Theorems \ref{th:exp bound 2} and \ref{th:abs 2} apply if $\left|\frac{w}{k\zeta}\right|<\frac{\epsilon_c}{8}$. But, by using \eqref{eq:V,iA petit 2'}, we can show that the result of these Theorems stay true if $\left|\frac{w}{k\zeta}\right|<\frac{\epsilon_c}{6}$.

 \smallskip
 
 \item If $\theta>0$, we can replace the assumption $V_c:\ch^1\rightarrow\ch^{-1}$ is compact by $V_c$ $\Delta$-compact and $\jap{q}^{\rho_{sr}}v\cdot\nabla \tilde{V}_{sr}$ bounded with the same result.
 
 \smallskip
 
\item If $\theta\leq 0$, $W_{\zeta \theta}$ is not $\Delta$-compact. Therefore Corollary \ref{c:FH} does not apply in this case.

\smallskip

 \item Making use the specific form of the potential, the absence of positive eigenvalue for $H$ was proved in \cite{JMb} if $\zeta>1$ and $\theta>1/2$.
 
 \smallskip
 
\item If $2\geq\zeta+\theta\geq3/2$, the regularity required by Theorem \ref{th:exp bound 2} is not granted. However we can prove the sub-exponential bounds along the lines of the proof of  \cite[Proposition 3.2]{JMb}.

\smallskip

\item Remark that, in \cite{Wh}, Schr\"odinger operators with oscillating potentials are studied, and it was used that potentials are central. But in our case, we do not suppose that other parts of the potential are central.
\end{enumerate}

\begin{proof}[Proposition \ref{prop:oscillant}]
Let $u\in\cu$ a bounded vector field. Suppose that $\zeta+\theta>1$.
Let $\tilde{\kappa}\in C^\infty_c(\R,\R)$ such that $\tilde{\kappa}(|x|)=0$ if $|x|\geq1$, $\tilde{\kappa}=1$ on $[-1/2,1/2]$ and $0\leq \tilde{\kappa}\leq 1$. So, we can observe that $(1-\tilde{\kappa}(|x|))(1-\kappa(|x|))=(1-\kappa(|x|))$ for all $x\in\R^\nu$.
For $\gamma\in\R$, let 

\begin{equation}\label{eq:tilde W}
\tilde{W}_{\zeta \gamma}(x)=w(1-\tilde{\kappa}(|x|))\frac{\cos(k|x|^\zeta)}{| x|^\gamma}.
\end{equation}
For $x\in\R^\nu$,
\begin{equation}\label{eq:W oscillant}
W_{\zeta \theta}(x)=-(1-\kappa(|x|))\frac{1}{k\zeta}\frac{x}{|x|}\cdot\nabla\tilde{W}_{\zeta \gamma}(x)-(1-\kappa(|x|))\gamma\frac{1}{k\zeta|x|}\tilde{W}_{\zeta \gamma}(x)
\end{equation}
with $\gamma=\theta+\zeta-1>0$.
Thus, by writing $\nabla\tilde{W}_{\zeta \gamma}=[p,i\tilde{W}_{\zeta \gamma}]$, we can show that $W_{\zeta \theta}:\ch^1\rightarrow\ch^{-1}$ is compact. Remark that since $\tilde{W}_{\zeta \gamma}$ has the same form as $W_{\zeta \theta}$, by iterated this calculus, we can show that, if $\zeta>1$, for all $l\in\N$, for all $k\in\R$, $l\geq (k-\theta)(\zeta-1)^{-1}$, $\jap{p}^{-l}\jap{q}^kW_{\zeta \theta}\jap{p}^{-l}$ is bounded. Similarly, since the derivative of $v$ is bounded, by writing
\[v\cdot\nabla\tilde{V}_{sr}=div(v)\tilde{V}_{sr}-div(v\tilde{V}_{sr}),\]
 we can show that $v\cdot\nabla \tilde{V}_{sr}:\ch^1\rightarrow\ch^{-1}$ is compact. Therefore, by sum, the first point of Proposition \ref{prop:oscillant} is proved.

To prove the next point, in a first time, we can see that, by \cite[Lemma 5.4]{Ma1}, if $\zeta+\theta>2$, $W_{\zeta \theta}$ has enough regularity to satisfies assumptions of Theorem \ref{th:exp bound 2}. Similarly, since $qV_c:\ch^1\rightarrow\ch^{-1}$ is compact ($V_c$ is compactly suport), we can show that all the terms of the potential has enough regularity to satisfies assumptions of Theorem \ref{th:exp bound 2}. 

If $3/2\leq\zeta+\theta\leq2$, we will adapt the proof of \cite[Proposition 7.1]{JMb} to our context. In this proof, we can see that it is sufficient to prove that $(\Psi_\lambda,[V,iA_u]\Psi_\lambda)$ is uniformly bounded in $\lambda$ to prove the polynomial bounds.

Suppose that $2\geq\zeta+\theta\geq3/2$. Then $V:\ch^1\rightarrow\ch^{-1}$ is compact which implies that $\sigma_{ess}(H)=\sigma_{ess}(\Delta)=[0,+\infty)$. In particular, we can find $m>0$, as large as we want, such that $-m\notin\sigma(H)$. In particular, by the resolvent formula, $-m\notin\sigma(H(F))$.

Let $F$ as in \eqref{eq:premier F}.  Let $H_0(F)=e^{F(Q)}H_0 e^{-F(Q)}$. Remark that $F(x)$ and $\nabla F(x)$ is bounded uniformly with respect to $\lambda>1$. As in \cite{JMb}, $(\Psi_\lambda,[V-W_{\zeta \theta}-V_c,iA_u]\Psi_\lambda)$ is uniformly bounded in $\lambda$. Therefore, we have to show that $(\Psi_\lambda,[V_c+W_{\zeta \theta},iA_u]\Psi_\lambda)$ is uniformly bounded in $\lambda$. By pseudodifferential calculus, we can show that, for all $l\in\R$, $\jap{P}^{l+2}(m+H_0(F))^{-1}\jap{P}^{-l}$ is uniformly bounded in $\lambda$. Notice that $\jap{P}(m+H(F))^{-1}\jap{P}$  is uniformly bounded in $\lambda$. Moreover, for $\epsilon\in[0,1]$,\\
 $\jap{q}^\epsilon\jap{P}(m+H_0(F))^{-1}\jap{P}\jap{q}^{-\epsilon}$ is uniformly bounded in $\lambda$.

We can write 
 \begin{eqnarray*}
& & (\Psi_\lambda,[V_c,iA_u]\Psi_\lambda)\\
&=&\left((H(F)+m)\Psi_\lambda,(H(F)+m)^{-1}[V_c,iA_u](H(F)+m)^{-1}(H(F)+m)\Psi_\lambda\right).
 \end{eqnarray*}
 Since $\jap{p}(H(F)+m)^{-1}\jap{p}$ is uniformly bounded and since $V_c$ is compactly support, we can easily see that $(H(F)+m)^{-1}[V_c,iA_u](H(F)+m)^{-1}$ is uniformly bounded. Using that $(H(F)+m)\Psi_\lambda=(E+m)\Psi_\lambda$, this implies that $(\Psi_\lambda,[V_c,iA_u]\Psi_\lambda)$ is uniformly bounded in $\lambda$.
 
For $(\Psi_\lambda,[W_{\zeta \theta},iA_u]\Psi_\lambda)$, notice that in the expression of $[W_{\zeta \theta},iA_u]$ there is only terms of the form $qW_{\zeta \theta}\cdot u(p)$ and $W_{\zeta \theta}div(u)(p)$. For terms with $W_{\zeta \theta}div(u)(p)$, since $W_{\zeta \theta}:\ch^1\rightarrow\ch^{-1}$ is compact, we know that $(H(F)+m)^{-1}W_{\zeta \theta}div(u)(p)(H(F)+m)^{-1}$ is uniformly bounded.

For the other type of terms, we can write for $l>0$
\begin{eqnarray*}
& &(\Psi_\lambda,qW_{\zeta \theta}\cdot u(p)\Psi_\lambda)\\
&=&\biggl((H(F)+m)^l\Psi_\lambda,(H(F)+m)^{-l}qW_{\zeta \theta}\cdot u(p)(H(F)+m)^{-l}(H(F)+m)^l\Psi_\lambda\biggr)\\
&=&(E+m)^{2l} \left(\Psi_\lambda,(H(F)+m)^{-l}qW_{\zeta \theta}\cdot u(p)(H(F)+m)^{-l}\Psi_\lambda\right).
\end{eqnarray*}

In particular, we only have to show that for $l$ large enough, 
\[(H(F)+m)^{-l}qW_{\zeta \theta} u(p)(H(F)+m)^{-l}\] is uniformly bounded in $\lambda$. To do this, we will use the resolvent estimate and write for all $M\in\N^*$
\begin{eqnarray*}
& &(H(F)+m)^{-1}\\
&=&(H_0(F)+m)^{-1}+\sum_{k=1}^{M}(-1)^k\left((H_0(F)+m)^{-1}V\right)^k(H_0(F)+m)^{-1}\\
& &+(-1)^{M+1}\left((H_0(F)+m)^{-1}V\right)^{M+1}(H(F)+m)^{-1}\\
&=&(H_0(F)+m)^{-1}+\sum_{k=1}^{M}(-1)^k(H_0(F)+m)^{-1}\left(V(H_0(F)+m)^{-1}\right)^k\\
& &+(-1)^{M+1}(H(F)+m)^{-1}\left(V(H_0(F)+m)^{-1}\right)^{M+1}.
\end{eqnarray*}
Remark that, since 
\[\jap{q}^{\zeta+\theta-1}\jap{p}^{-1}\left(V_{sr}+v\cdot\nabla \tilde{V}_{sr}+V_c+W_{\zeta \theta}\right)\jap{p}^{-1}\]
 is bounded and since, for all $\epsilon\in[0,1]$, $\jap{q}^\epsilon\jap{P}(m+H_0(F))^{-1}\jap{P}\jap{q}^{-\epsilon}$ is uniformly bounded, we can write
\begin{eqnarray*}
& &(H(F)+m)^{-1}\\
&=&(H_0(F)+m)^{-1}+\sum_{k=1}^{M}(-1)^k\left((H_0(F)+m)^{-1}V_{lr}\right)^k(H_0(F)+m)^{-1}\\
& &+(-1)^{M+1}\left((H_0(F)+m)^{-1}V_{lr}\right)^{M+1}(H(F)+m)^{-1}+\jap{p}^{-1}\jap{q}^{1-\zeta-\theta}B_1\\
&=&(H_0(F)+m)^{-1}+\sum_{k=1}^{M}(-1)^k(H_0(F)+m)^{-1}\left(V_{lr}(H_0(F)+m)^{-1}\right)^k\\
& &+(-1)^{M+1}(H(F)+m)^{-1}\left(V_{lr}(H_0(F)+m)^{-1}\right)^{M+1}+B_2\jap{q}^{1-\zeta-\theta}\jap{p}^{-1}
\end{eqnarray*}
where $B_1,B_2$ are uniformly bounded in $\lambda$. Now, we will choose $M\in\N^*$ such that $(M+1)\rho_{lr}\geq \zeta+\theta-1$. By a simple computation, we can see that 
\begin{eqnarray*}
& &(-1)^{M+1}(H(F)+m)^{-1}\left(V_{lr}(H_0(F)+m)^{-1}\right)^{M+1}\jap{q}^{1-\zeta-\theta}\jap{p}^{-1}\text{ and }\\
& &\jap{p}^{-1}\jap{q}^{1-\zeta-\theta}(-1)^{M+1}\left((H_0(F)+m)^{-1}V_{lr}\right)^{M+1}(H(F)+m)^{-1}
\end{eqnarray*}
are uniformly bounded.

By taking the power $l>0$, we have
\begin{eqnarray*}
& &(H(F)+m)^{-l}\\
&=&\biggl((H_0(F)+m)^{-1}+\sum_{k=1}^{M}(-1)^k\left((H_0(F)+m)^{-1}V_{lr}\right)^k(H_0(F)+m)^{-1}\biggr)^l\\
& &+\jap{p}^{-1}\jap{q}^{1-\zeta-\theta}B'_1\\
&=&\biggl((H_0(F)+m)^{-1}+\sum_{k=1}^{M}(-1)^k(H_0(F)+m)^{-1}\left(V_{lr}(H_0(F)+m)^{-1}\right)^k\biggr)^l\\
& &+B'_2\jap{q}^{1-\zeta-\theta}\jap{p}^{-1}\\
\end{eqnarray*}
with $B'_1,B'_2$ are uniformly bounded in $\lambda$.

Notice that $V_{lr}\jap{p}^{-2}\jap{p}^2(H_0(F)+m)^{-1}$ is bounded. By a simple computation, we can remark that $\jap{q}[V_{lr},\jap{p}^{-2}]$ is bounded. In particular, we can write 
\[V_{lr}(H_0(F)+m)^{-1}=\jap{p}^{-2}V_{lr}\jap{p}^2(H_0(F)+m)^{-1}+\jap{q}^{-1}B_3\]
 with $B_3$ uniformly bounded. Similarly, we can write
 \[(H_0(F)+m)^{-1}V_{lr}=(H_0(F)+m)^{-1}\jap{p}^2V_{lr}\jap{p}^{-2}+B_4\jap{q}^{-1}\]
 with $B_4$ uniformly bounded. Repeating this computation, we can see that
 \begin{eqnarray*}
& &(H(F)+m)^{-l}\\
&=&\jap{p}^{-2l}B'_5+\jap{p}^{-1}\jap{q}^{-1}B'_3+\jap{p}^{-1}\jap{q}^{1-\zeta-\theta}B'_1\\
&=&B'_6\jap{p}^{-2l}+B'_4\jap{q}^{-1}\jap{p}^{-1}+B'_2\jap{q}^{1-\zeta-\theta}\jap{p}^{-1}
\end{eqnarray*}
with $(B'_k)_{k=1,\cdots,6}$ uniformly bounded in $\lambda$.
Thus,
\begin{eqnarray*}
& &(H(F)+m)^{-l} qW_{\zeta \theta} u(p)(H(F)+m)^{-l}\\
&=&\left(B'_6\jap{p}^{-2l}+B'_4\jap{q}^{-1}\jap{p}^{-1}+B'_2\jap{q}^{1-\zeta-\theta}\jap{p}^{-1}\right)qW_{\zeta \theta} u(p)\\
& &\left(\jap{p}^{-2l}B'_5+\jap{p}^{-1}\jap{q}^{-1}B'_3+\jap{p}^{-1}\jap{q}^{1-\zeta-\theta}B'_1\right)
\end{eqnarray*}
Since $\jap{p}^{-1} qW_{\zeta \theta} \jap{p}^{-1}\jap{q}^{1-\zeta-\theta}$ is bounded, we can write 
\[
(H(F)+m)^{-l} qW_{\zeta \theta} u(p)(H(F)+m)^{-l}=B'_6\jap{p}^{-2l}qW_{\zeta \theta} u(p)\jap{p}^{-2l}B'_5+B
\]
where $B$ is uniformly bounded in $\lambda$. By taking $l$ large enough such that $\jap{p}^{-2l}qW_{\zeta \theta} \jap{p}^{-2l}$ is bounded, we show that $(H(F)+m)^{-l} qW_{\zeta \theta} u(p)(H(F)+m)^{-l}$ is uniformly bounded in $\lambda$. This implies that $(\Psi_\lambda,[V,iA_u]\Psi_\lambda)$ is uniformly bounded in $\lambda$ and we infer the polynomial bounds. Using a similar proof with $F$ as in \eqref{eq:deuxieme F}, we prove the sub-exponential bounds (point (2) of Proposition \ref{prop:oscillant}).

To prove that this sub-exponential bounds are unlimited, we only have to show that the Mourre estimate is true on all compact subset of $(0,+\infty)$. Let $\chi\in C^\infty_c$ supported on a compact subset of $(0,+\infty)$.
Suppose that $\zeta+\theta>3/2$. Then there is $a>0$ such that:

\begin{eqnarray*}
\chi(H)[H,iA_u]\chi(H)&=&\chi(H_0)[\Delta,iA_u]\chi(H_0)+(\chi(H)-\chi(H_0))[\Delta,iA_u]\chi(H_0)\\
& &+\chi(H_0)[\Delta,iA_u](\chi(H)-\chi(H_0))\\
& &+(\chi(H)-\chi(H_0))[\Delta,iA_u](\chi(H)-\chi(H_0))\\
& &+\chi(H)[W_{\zeta \theta},iA_u]\chi(H)\\
&\geq&a\chi(H_0)^2+(\chi(H)-\chi(H_0))[\Delta,iA_u]\chi(H_0)\\
& &+\chi(H_0)[\Delta,iA_u](\chi(H)-\chi(H_0))\\
& &+(\chi(H)-\chi(H_0))[\Delta,iA_u](\chi(H)-\chi(H_0))\\
& &+\chi(H)[W_{\zeta \theta},iA_u]\chi(H).
\end{eqnarray*}
Remark that since $H$ is a compact perturbation of $H_0=\Delta$, $(\chi(H)-\chi(H_0))$ is compact on $\ch^1$ to $\ch^{-1}$. In particular the second, the third and the fourth terms of the r.h.s. of the previous inequality are compact. Moreover, since $\chi(H)(H+m)^l$ is bounded for all $l>0$, using that $\jap{P}^{-1}QW_{\zeta \theta}(Q)\jap{P}^{-1}\jap{Q}^{1-\theta-\zeta}$ is compact if $\zeta+\theta>3/2$ and using a similar proof than in the previous point, we can show that $\chi(H)[W_{\zeta \theta},iA_u]\chi(H)$ is compact. So there is $a>0$ and $K$ compact such that
\begin{equation}\label{eq: Mourre est compact}
\chi(H)[H,iA_u]\chi(H)\geq a\chi(H)^2+K.
\end{equation}
Let $\lambda_0\in (0,+\infty)$ and $I$ an open real set containing $\lambda_0$ such that the closure of $I$ is included in $(0,+\infty)$. Take $\chi$ as previously such that $\chi=1$ on $I$. Remark that $\chi(H)E(I)=E(I)\chi(H)=E(I)$. Thus, by multiplying on the left and on the right of \eqref{eq: Mourre est compact} by $E(I)$, we obtain the Mourre estimate at $\lambda_0$ w.r.t. the conjugate operator $A_u$(point (3) of Proposition \ref{prop:oscillant}).

Now, suppose that $\zeta+\theta\geq2$. By Corollary \ref{c:deriv sr} and \eqref{eq:W oscillant}, we already know that if $\||q|\tilde{W}_{\zeta \gamma}\|$ is small enough, then $W_{\zeta \theta}$ satisfies \eqref{eq:V,iA petit 1} and \eqref{eq:V,iA petit 2} and Theorems \ref{th:abs 1} and \ref{th:abs 2} apply. Thus we only have to show that this norm is small enough.

Suppose that $\zeta+\theta=2$ and $\zeta>1$. 
Since $\jap{q}^{1+\rho_{sr}}\tilde{V}_{sr}$ is bounded, we can use Corollary \ref{c:deriv sr} on $v\cdot\nabla\tilde{V}_{sr}$.
Remark that 
\[\left\|\frac{|q|}{k\zeta }(1-\kappa(|q|))\tilde{W}_{\zeta \gamma}(q)\right\|=\left|\frac{w}{k\zeta}\right|.\]
In particular, if $\left|\frac{w}{k\zeta}\right|<\frac{\epsilon_c}{8}$, for all $C>0$, we can find $\epsilon>0$ small enough such that 
\[-\left(C\epsilon+4\left|\frac{w}{k\zeta}\right|\right)-4\left|\frac{w}{k\zeta}\right|+\epsilon_c-2>-2.\]
Therefore, by Corollary \ref{c:deriv sr}, Theorem \ref{th:abs 2} applies and we prove this part of the Proposition. Using the assumption \eqref{eq:V,iA petit 2'} instead of \eqref{eq:V,iA petit 2} in Theorem \ref{th:abs 2}, we can remark that it suffices to have $\left|\frac{w}{k\zeta}\right|<\frac{\epsilon_c}{6}$.

Suppose that $\zeta+\theta>2$. In this case, $\gamma=\zeta+\theta-1>1$. In particular, $|q|\tilde{W}_{\zeta \gamma}$ vanishes at infinity. So, for all $\epsilon>0$, we can find $\tilde{\chi}\in C^\infty_c$, such that $\tilde{\chi}(t)=1$ if $|t|<1$, $0\leq\tilde{\chi}\leq1$ and $\|(1-\tilde{\chi}(q))|q|W\|<\epsilon$.
Thus, by Corollary \ref{c:deriv sr}, for $\epsilon$ small enough, \eqref{eq:V,iA petit 2} is satisfied and Theorem \ref{th:abs 2} applies (point (5) of Proposition \ref{prop:oscillant}).
\qed
\end{proof}

\subsection{A potential with high oscillations}\label{ss:oscillant exp}

Let 
\[V(x)=w(1-\kappa(|x|))\exp(3|x|/4)\sin(\exp(|x|))\]
with $w\in\R$, $\kappa\in C^\infty_c(\R,\R)$, $0\leq\kappa\leq1$ and $\kappa(|x|)=1$ if $|x|<1$.

For all $w\in\R$, we have the following:
\begin{lemma}\label{l:ex expo}
Let $V$ as previously. Then
\begin{enumerate}
\item $V:\ch^1\rightarrow\ch^{-1}$ is compact;

\smallskip

\item for all $u\in C^\infty$ bounded with all derivatives bounded, $V\in C^\infty(A_u,\ch^1,\ch^{-1})$;

\smallskip

\item $H=\Delta+V$ has no positive eigenvalues.
\end{enumerate}
\end{lemma}

Remark that since $V$ is not $\Delta$-compact and since $V$ is not in $C^1(A_D,\ch^1,\ch^{-1})$ (see \cite[Lemma 5.6]{Ma1}), we can not apply Corollary \ref{c:FH} and we can not use the Mourre Theorem with $A_D$ as conjugate operator. But, by Corollary \ref{c:resolvent}, we can prove that 
\[\lambda\mapsto R(\lambda\pm i0)\]
are of class $C^\infty$ on $(0,+\infty)$.

\begin{proof}[ Lemma \ref{l:ex expo}]
%
Let $\tilde{\kappa}\in C^\infty_c(\R,\R)$ such that $\tilde{\kappa}(|x|)=0$ if $|x|\geq1$, $\tilde{\kappa}=1$ on $[-1/2,1/2]$ and $0\leq \tilde{\kappa}\leq 1$.
Let $\tilde{V}(x)=(1-\tilde{\kappa}(|x|))\cos(e^{|x|})$. Then, we have
\[(1-\kappa(|x|))\nabla \tilde{V}(x)=-(1-\kappa(|x|))\frac{x}{|x|}\exp(|x|)\sin(\exp(|x|)).\]
So 
\begin{eqnarray*}
xV(x)&=&-w|x|(1-\kappa(|x|))\exp(-|x|/4)\nabla\tilde{V}(x)\\
&=&-w|x|\nabla\left((1-\kappa(|x|))\exp(-|x|/4)\tilde{V}(x)\right)\\
& &-wx\kappa'(|x|)\exp(-|x|/4)\tilde{V}(x)-\frac{w}{4}x(1-\kappa(|x|))\exp(-|x|/4)\tilde{V}(x) .
\end{eqnarray*}
As previously, by writing $\nabla\tilde{V}=[p,iV]$, we can show that $V:\ch^1\rightarrow\ch^{-1}$ is compact. Moreover, by \cite[Lemma 5.6]{Ma1}, we already know that $V\in C^{\infty}(A_u,\ch^1,\ch^{-1})$ for all $u\in\cu$ bounded. Those implies that Theorem \ref{th:exp bound 2} applies. Moreover, since $|q|\tilde{V}$ vanishes at infinity, as previously, for all $\epsilon>0$, we can find $\tilde{\chi}\in C^\infty_c$, such that $\tilde{\chi}(t)=1$ if $|t|<1$, $0\leq\tilde{\chi}\leq1$ and $\|(1-\tilde{\chi}(q))|q|W\|<\epsilon$. Thus, by Corollary \ref{c:deriv sr}, we can find $\epsilon>0$ small enough such that \eqref{eq:V,iA petit 2} is true. Therefore Theorem \ref{th:abs 2} applies and $H=\Delta+V$ has no positive eigenvalues.
 \qed
\end{proof}

\appendix 

\section{The Helffer-Sj\"ostrand formula}\label{s: H-S formula}

Let $T$ and $B$ two self-adjoint operators. Let $ad^1_B(T)=[T,B]$ be the commutator. We denote $ad^p_B(T)=[ad^{p-1}_B(T),B]$ the iterated commutator. Furthermore, if $T$ is bounded, $T$ is of class $C^k(B)$ if and only if for all $0\leq p\leq k$, $ad^p_B(T)$ is bounded.

\begin{proposition}[\cite{DG} and \cite{Mo}]
Let $\varphi\in \cs^\rho$, $\rho\in\R$. For all $l\in\R$, there is a smooth function $\varphi^\C:\C\rightarrow\C$, called an almost-analytic extension of $\varphi$, such that :
\begin{equation}
\varphi^\C_{|\R}=\varphi \qquad \frac{\partial \varphi^\C}{\partial \bar{z}}=c_1 \langle\Re(z)\rangle^{\rho-1-l}|\Im(z)|^l
\end{equation}
\begin{equation}
supp \varphi^\C\subset \{x+iy\|y|\leq c_2\langle x\rangle\}
\end{equation}
\begin{equation}
\varphi^\C(x+iy)=0, \text{ if }x\notin supp(\varphi)
\end{equation}
for constant $c_1$ and $c_2$ depending of the semi-norms of $\varphi$.
\end{proposition}

\begin{theorem}[\cite{GoJ} and \cite{Mo}]
Let $k\in\N^*$ and $T$ a bounded operator in $C^k(B)$. Let $\rho<k$ and $\varphi\in\cs^\rho$. We have

\begin{equation}\label{eq: H-S formula}
[\varphi(B),T]=\sum^{k-1}_{j=1}\frac{1}{j!}\varphi^{(j)}(B)ad^j_B(T)+\frac{i}{2\pi}\int_\C \frac{\partial \varphi^\C}{\partial\bar{z}}(z-B)^{-k}ad^k_B(T)(z-B)^{-1} dz\wedge d\bar{z}
\end{equation}
\end{theorem}

In the general case, the rest of the previous expansion is difficult to calculate. So we will give an estimate of this rest.

\begin{proposition}[\cite{GoJ} and \cite{Mo}]\label{prop: estimate H-S}
Let $T\in C^k(A)$ be a self-adjoint and bounded operator. Let $\varphi\in\cs^\rho$ with $\rho<k$. Let 
\[I_k(\varphi)=\int_\C \frac{\partial \varphi^\C}{\partial\bar{z}}(z-B)^{-k}ad^k_B(T)(z-B)^{-1} dz\wedge d\bar{z}\]
 be the rest of the development of order $k$ in \eqref{eq: H-S formula}. Let $s,s'>0$ such that $s'<1$, $s<k$ and $\rho+s+s'<k$. Then $\langle B \rangle^s I_k(\varphi)\langle B\rangle^{s'}$ is bounded.
\end{proposition}

In particular, if $\rho<0$, and if we choose $s'$ near $0$, we have $\langle B \rangle^s I_k(\varphi)\langle B\rangle^{s'}$ bounded, for all $s<k-s'-\rho$.

{\bf Acknowledgements.} I thank my doctoral supervisor, Thierry Jecko, for fruitfull discussions and comments.

\bibliographystyle{alpha}
\bibliography{bibliographie}

\end{document}